\documentclass[hidelinks,onefignum,onetabnum]{./dependencies/siamart220329}

\usepackage{lipsum}
\usepackage{amsfonts}
\usepackage{graphicx}
\usepackage{epstopdf}
\usepackage{algorithmic}
\ifpdf
  \DeclareGraphicsExtensions{.eps,.pdf,.png,.jpg}
\else
  \DeclareGraphicsExtensions{.eps}
\fi

\newsiamremark{remark}{Remark}
\newsiamremark{hypothesis}{Hypothesis}
\crefname{hypothesis}{Hypothesis}{Hypotheses}
\newsiamthm{claim}{Claim}

\headers{An IF-IPM for LO with High Adaptability to QC}{M. Mohammadisiahroudi et al.}

\title{An Inexact Feasible  Interior Point Method for Linear Optimization with High Adaptability to Quantum Computers \thanks{Submitted to the editors DATE.
\funding{This work is supported by Defense Advanced Research Projects Agency as a part of the project W911NF2010022: {\em The Quantum
Computing Revolution and Optimization: Challenges and Opportunities}.}
}
}

\author{Mohammadhossein Mohammadisiahroudi\thanks{Quantum Computing and Optimization Lab, Industrial and System Engineering Department, Lehigh University, Bethlehem, PA, USA \email{(mom219@Lehigh.edu, fakhimi@lehigh.edu, zew220@lehigh.edu, terlaky@lehigh.edu)}. }
\and Ramin Fakhimi\footnotemark[2]
\and Zeguan Wu\footnotemark[2]
\and Tam\'as Terlaky\footnotemark[2]}

\usepackage{amsopn}

\usepackage{dependencies/dependencies}
\usepackage{booktabs}
\graphicspath{{images/}}
\usepackage{multirow}
\ifpdf
\hypersetup{
  pdftitle={An Inexact Feasible  Interior Point Method for Linear Optimization with High Adaptability to Quantum Computers},
  pdfauthor={M. Mohammadisiahroudi, 
R. Fakhimi, Z. Wu, and 
T.  Terlaky}
}
\fi

\usepackage{comment}

\begin{document}

\maketitle
\begin{abstract}
The use of quantum computing to accelerate complex optimization problems is a burgeoning research field.
This paper applies Quantum Linear System Algorithms (QLSAs) to Newton systems within Interior Point Methods (IPMs) to take advantage of quantum speedup in solving Linear Optimization (LO) problems.
Due to their inexact nature, QLSAs can be applied only to inexact variants of IPMs.
Existing IPMs with inexact Newton directions are infeasible methods due to the inexact nature of their computations.
This paper proposes an Inexact-Feasible IPM (IF-IPM) for LO problems, using a novel linear system to generate inexact but feasible steps.
We show that this method has $\Ocal(\sqrt{n}L)$ iteration complexity, analogous to the best exact IPMs, where $n$ is the number of variables and $L$ is the binary length of the input data.
Moreover, we examine how QLSAs can efficiently solve the proposed system in an iterative refinement (IR) scheme to find the exact solution without excessive calls to QLSAs.
We show that the proposed IR-IF-IPM can also be helpful in mitigating the impact of the condition number when a classical iterative method, such as a Conjugate Gradient method, or a quantum solver is used at iterations of IPMs.
After applying the proposed IF-IPM to the self-dual embedding formulation, we implement it and present illustrative numerical results. 
\end{abstract}



\section{Introduction}
Recently, major investments have been going into building efficient quantum computers and solving crucial real-world problems.
Starting with Deutsch’s method~\cite{Deutsch1985_Quantum}, quantum computing shows exponential speed-up compared to conventional computers in solving some challenging mathematical problems, such as the integer factorization problem~\cite{Shor1994_algorithms} and unstructured search problem~\cite{Grover1996_fast}.
Due to the wide range of applications of mathematical optimization problems and their intrinsic challenges, many researchers have attempted to develop quantum optimization algorithms, such as the Quantum Approximation Optimization Algorithm (QAOA) for quadratic unconstrained binary optimization
\cite{farhi2014quantum}, quantum subroutines for the simplex method \cite{nannicini2022fast}, Quantum Multiplicative Weight Update Method (QMWUM) for semidefinite optimization (SDO) \cite{augustino2023solving}, and Quantum Interior Point Methods (QIPMs) for linear optimization (LO) problems~\cite{augustino2021quantum,Casares2020_quantum,kerenidisParkas2020_quantum,mohammadisiahroudi2022efficient}. 

QIPMs are structurally analogous to classical Interior Point Methods (IPMs) that use Quantum Linear System Algorithms (QLSAs) to solve the Newton system at each iteration.
Inexact IPMs benefit from the inexact solutions provided by QLSAs.
Mohammadisiahroudi et al.~\cite{mohammadisiahroudi2022efficient} proposed an Inexact Infeasible IPM (II-IPM) to cope with the inexactness of the solution of the Newton system.
Motivated by the efficient use of QLSA in IPMs, we develop an Inexact Feasible IPM (IF-IPM) using a novel system.
The proposed IF-IPM starts from a feasible interior point, and the iterates remain in the interior of the feasible region even with an inexact solution to the proposed system. 
In the quantum version of the proposed IF-IPM, we efficiently use QLSA to accelerate the solution of LO problems.
First, we define the LO problem.
\begin{definition}[Linear Optimization Problem: Standard Form]\label{def: LO}
For vectors $b\in \mathbb{R}^m$, $c\in \mathbb{R}^n$, and matrix $A\in \mathbb{R}^{m\times n}$ with $\rank(A)=m$, we define the primal-dual pair of LO problems as:

\begin{center}
\begin{minipage}{.4\linewidth}
\renewcommand{\theequation}{\text{P}}
\begin{equation}
 \begin{aligned}
    \min\  c^T&x, \\
    {\rm s.t. }\;\;
    Ax &= b, \\
    x &\geq 0,
    \end{aligned}
 \label{P}
\end{equation}
\end{minipage}%
\begin{minipage}{.4\linewidth}
\renewcommand{\theequation}{\text{D}}
\begin{equation}
 \begin{aligned}
    \max \  b^Ty,\ \ & \\
    {\rm s.t. }\;\;
    A^Ty +&s = c,\\
    &s \geq 0,
    \end{aligned}
 \label{D}
\end{equation}
\end{minipage}
\end{center}

\noindent
where $x\in \Rmbb^n$ is the vector of primal variables, and $y\in\Rmbb^m$, $s\in\Rmbb^n$ are vectors of dual variables.
Problem $(P)$ is called the primal problem and $(D)$ is called the dual problem. 
\end{definition}
As we can see in the definition, a common assumption is that $A$ has full row rank. LO problems can also be presented in another form, known as the canonical form.
\begin{definition}[Linear Optimization Problem: Canonical Form]\label{def: LO: conocical}

\vspace{-1em}
\centering
\begin{minipage}{.4\linewidth}
\renewcommand{\theequation}{\text{$\rm P^\prime$}}
\begin{equation}
 \begin{aligned}
    \min\  c'^T&x, \\
    {\rm s.t. }\;\;
    A'x &\geq b', \\
    x &\geq 0,
    \end{aligned}
 \label{Pprime}
\end{equation}
\end{minipage}%
\begin{minipage}{.4\linewidth}
\renewcommand{\theequation}{\text{$\rm D^\prime$}}
\begin{equation}
 \begin{aligned}
    \max \  b'^Ty,\ \ & \\
    {\rm s.t. }\;\;
    A'^Ty &\leq c',\\
    &y \geq 0.
    \end{aligned}
 \label{Dprime}
\end{equation}
\end{minipage}
\end{definition}
The standard and canonical forms are equivalent, and one can derive both forms for any LO problem. 
By finding basic variables for primal problem \eqref{P}, we can derive the canonical form from the standard form. 
In this case, the canonical form has $n'=n-m$ variables and $m'= m$ constraints. 
Observe that matrix $A'$ does not necessarily have full row rank, and possibly one has $m'>n'$.
An LO problem in canonical form can be transformed to standard form just by adding slack variables. We are going to use both forms in this paper. 
However, the default is the standard form, and the reader is notified when the canonical form is used. 
Using the standard form of LO problems, the set of feasible primal-dual solutions is defined as
\begin{equation*}
\mathcal{PD}=\left\{(x,y,s)\in\mathbb{R}^{n} \times \mathbb{R}^m\times\mathbb{R}^n |\ Ax=b,\ A^Ty+s=c, \ (x,s)\geq0\right\}.
\end{equation*}
Then, the set of all feasible interior solutions is
\begin{equation*}
\mathcal{PD}^0=\left\{(x,y,s)\in\mathcal{PD}\ |\ (x,s)>0\right\}.
\end{equation*}
By the Strong Duality theorem, all optimal solutions, if there exist any, belong to the
set $\mathcal{PD}^*$ defined as
\begin{equation*}
\mathcal{PD}^*=\left\{(x,y,s)\in\mathcal{PD} \ | \ x^Ts=0\right\}.
\end{equation*}
Let $\zeta \geq 0$, then the set of $\zeta$-optimal solutions can be defined as
\begin{equation*}
\mathcal{PD}_{\zeta} = \left\{(x,y,s)\in\mathcal{PD} \ \Big|\ \frac{x^Ts}{n}\le \zeta \right\}.
\end{equation*}
Dantzig's Simplex method was the first efficient algorithm to solve LO problems~\cite{dantzig2016linear}.
Klee and Minty~\cite{Klee1972_good} showed that Simplex methods have an exponential worst-case complexity. 
Khachiyan~\cite{Khachiyan1979_polynomial} proposed the Ellipsoid method for solving LO problems with integer input data and presented the first polynomial time algorithm for LO.
Nonetheless, the Ellipsoid method was practically less efficient than simplex methods.
Karmarkar~\cite{Karmarkar1984_New} developed an Interior Point Method (IPM) for solving LO problems with polynomial time complexity.
Following his work, many theoretically and practically efficient IPMs were developed, see e.g.,~\cite{ Roos2005_Interior, Terlaky2013_Interior,Wright1997_Primal}.

A feasible IPM converges to an optimal solution by starting from an interior point and following the so-called central path~\cite{Roos2005_Interior}.
Most of the efficient IPMs are primal-dual methods, meaning that they attempt to satisfy the optimality condition while maintaining both primal and dual feasibility.
To develop our IF-IPM, we use the primal-dual path-following feasible IPM paradigm. 
Assuming that $\mathcal{PD}^0\not = \emptyset$, the central path is defined as
\begin{equation*}
    \mathcal{CP}=\Big\{(x,y,s)\in\mathcal{PD}^0\ \big| \ x_is_i=\mu\  \text{ for }\ i\in\{1,\dots,n\}; \text{ for }\ \mu>0\Big\}.
\end{equation*}
For any $\theta\in[0,1)$, a neighborhood of the central path can be defined as 
\begin{align*}
    \mathcal{N}(\theta)=\Big\{(x,y,s)\in\mathcal{PD}^0\ \big| \|XSe-\mu e\|_2\leq \theta \mu \Big\},
\end{align*}
where $e\in\Rmbb^n$ is the all one vector, and $X$ and $S$ are diagonal matrices of $x$ and $s$, respectively. 
Throughout this paper, we use $\|M\|$ as the 2-norm of matrix $M$, and $\|M\|_F$
as the Frobenius norm of $M$. 
We also use $\tilde{\Ocal}$, which suppresses the polylogarithmic factors in the ``Big-O'' notation.
Subscripts of $\tilde{\Ocal}$ indicate the parameters/quantities occurring in the suppressed polylogarithmic factors.

IPMs can be categorized into two groups: Feasible IPMs and Infeasible IPMs. 
Feasible IPMs (F-IPM) require an initial feasible interior point as a starting point. 
They frequently employ a self-dual embedding (SDE) formulation of the LO problem, where a feasible interior solution can be easily constructed~\cite{Roos2005_Interior}.
Instead, Infeasible IPMs (I-IPMs) start with an infeasible but strictly positive solution.
Theoretical analysis shows that the best time complexity of F-IPMs for LO problems is $\Ocal(\sqrt{n}L)$ where $L$ is the binary length of the input data.
On the other hand, the best time complexity of I-IPMs for LO problems is $\Ocal(nL)$. 
Despite the theoretical advantage of F-IPMs over I-IPMs, both feasible and infeasible IPMs can solve LO problems efficiently in practice~\cite{Wright1997_Primal}.

Recent studies have considered the convergence of IPMs with inexact search directions because of the inherent inexactness of limited, finite precision arithmetic in classical computers. 
First, Mizuno and his colleagues did a series of research on the convergence of II-IPMs~\cite{ MizunoJarre1999_Global, Freund1999_Convergence}. 
Later, Baryamureeba and Steihaung~\cite{Baryamureeba2006_convergence} proved the convergence of a variant of the I-IPM of~\cite{Kojima1993_primal} with an inexact Newton step.
Korzak~\cite{Korzak2000_Convergence} also showed that his proposed II-IPM has a polynomial time complexity.

Several authors studied the use of Preconditioned Conjugate Gradient method (PCGM) in II-IPMs~\cite{Gondzio2009_Convergence, Monteiro2003_Convergence}.
AL-Jeiroudi and Gondzio~\cite{Gondzio2009_Convergence} used the I-IPM of~\cite{Wright1997_Primal} while solving the Augmented system (AS) by a PCGM.
Monteiro and O'Neal ~\cite{Monteiro2003_Convergence} applied a PCGM to the Normal Equation System (NES).
Bellavia~\cite{Bellavia1998_Inexact} studied the convergence of the II-IPM for general convex optimization problems.
Zhou and Toh~\cite{Zhou2004_Polynomiality} developed an II-IPM for the Semidefinite Optimization (SDO) problems.
The best bound for the number of iterations of II-IPMs for LO problems is $\Ocal(n^2L)$.

All proposed inexact versions of IPMs are also infeasible since the inexact solutions to conventional formulations of Newton systems, such as NES and AS, lead to infeasibility. 
Gondzio \cite{gondzio2013convergence} showed that if Newton systems arising in IPMs can be solved inexactly such that feasibility is maintained, IPMs can leverage the best iteration complexity $\Ocal(\sqrt{n}\log(\frac{1}{\epsilon}))$ for quadratic optimization.
To exploit this favorable complexity of feasible IPMs, we introduce a new form of the Newton system for finding an inexact but feasible step and develop an IF-IPM. 
We prove the polynomial-time convergence of the proposed IF-IPM and show a polynomial speedup w.r.t the dimension of the problem, using QLSA to solve the novel system, compared to previous classical and quantum IPMs.
As the proposed IF-QIPM can handle the error and noise of contemporary quantum computers, and the novel system has a better condition number bound, the proposed approach shows a high level of adaptability to quantum computing, which is evidenced by the complexity advantage of the proposed IR-IF-QIPM compared to previous quantum and classical IPMs.
We also explore the efficiency of the proposed algorithm using classical iterative solvers like CGM.

This paper is structured as follows. 
In Section~\ref{sec: IF-IPM}, a novel system is proposed to produce an inexact but feasible Newton step along with developing IF-IPM. 
The characteristics of the novel system are analyzed and compared to other forms of the Newton system in Section~\ref{sec: anal system}. 
Section~\ref{sec: IF-QIPM} explores how to use a QLSA to solve the novel system in order to develop an IF-QIPM.  
In Section~\ref{sec: IF-IPM CGM}, we present the classical counterpart of the proposed IF-IPM using CGMs.
An iterative refinement scheme is designed in Section~\ref{sec: IR} to mitigate the impact of increasing condition number and precision on the total complexity of both IF-QIPM, and also for IF-IPM with CGM.
We adapt the proposed IF-IPM to the SDE formulation of LO problems in Section~\ref{sec: self-dual}. 
Computational experiments are presented in Section~\ref{sec: numerical}, and conclusions are provided in Section~\ref{sec: conclusion}.

\section{Inexact Feasible IPM} \label{sec: IF-IPM}
F-IPMs have the best computational complexity, which can be further enhanced by solving the Newton system with QLSAs at each iteration.
In order to investigate this opportunity, we propose a novel IF-IPM.
At each step of IPMs, there are three choices of linear systems to calculate the Newton step: Augmented system (AS), Normal Equation System (NES), and Full Newton System (FNS).
Solving any of these three systems inexactly leads to residuals in the primal and/or dual feasibility equations.
In this paper, we develop an IF-IPM to avoid the infeasibility caused by residuals.
By constructing a new system that offers a primal-dual feasible step based on a basis of orthogonal subspaces, we avoid the additional cost associated with infeasible IPMs.
With this structure, we utilize short-step feasible IPMs with inexact Newton steps.
%
\subsection{Orthogonal Subspaces System}
For a feasible interior solution $(x,y,s)\in \mathcal{PD}^0$, the Newton system is defined as
\begin{equation*}\label{eq:newton system}\tag{FNS}
\begin{aligned}
    A\Delta x=0,\  A^T \Delta y +\Delta s=0,\ X\Delta s+S\Delta x=\beta \mu e-Xs,
\end{aligned} 
\end{equation*}
where $\beta\in [0,1]$ is the reduction parameter, $\mu=\frac{x^Ts}{n}$, $X= \text{diag}(x)$, and $S= \text{diag}(s)$.

As we can see in \eqref{eq:newton system}, to calculate a feasible step, $\Delta x$ should be in the null space of $A$ and $\Delta s$ should be in the row space of $A$. To maintain feasibility, another way of reformulating the Newton system is to express $\Delta x$ and $\Delta s$ as linear combinations of bases for the null space and the range space of $A$, respectively. In this section, we show how this reformulation can help to maintain feasibility even with errors in the solution of the system.

Let $a_{i}$ be the $i$\textsuperscript{th} column of the matrix $A$.
We define index set $B \subseteq \{1,\dots,n\}$ as the index set of $m$ linearly independent columns of $A$, and $A_B=[a_i]_{i\in B}$.
Since $A$ has full row rank, $m$ linearly independent columns of $A$ do exist.
Thus, matrix $A_B$ is non-singular, and  $A_B^{-1}$ as the inverse of $A_B$ exists.
For ease of exposition, we may assume w.l.g. that the matrix $A_B$ is formed by the first $m$ columns of matrix $A$.
By pivoting on matrix $A=\begin{bmatrix}A_B & A_N\end{bmatrix}$, we can construct matrix $\begin{bmatrix}I& A_B^{-1}A_N\end{bmatrix}\in \mathbb{R}^{m\times n}$. 

We also construct matrices $V \in \mathbb{R}^{n\times (n-m)}$ and $W \in \mathbb{R}^{n\times m}$ as follows
\begin{equation*}
    V=\begin{bmatrix}
A_B^{-1}A_N\\
-I
\end{bmatrix}, \qquad W= A^T.
\end{equation*}
Calculating $V$ requires $\Ocal(mn^2)$ arithmetic operations. 
We can avoid this computational cost if the LO is defined in canonical form. 
In practice, most of the constraints are inequalities, and their slack variables can be used in basis $A_B$, which reduces this preprocessing cost.
In this paper, we neglect the preprocessing cost, since one can avoid preprocessing by using the following reformulation.
\begin{equation*} 
    \begin{aligned}
    \min\  c^Tx,\ {\rm s.t. }\;\; Ax +s'= b,\ -Ax +s''= -b, \  x,s',s'' \geq 0.
    \end{aligned}
\end{equation*}
In this formulation, $s'$ and $s''$ form a basis, and matrix $V$ can be constructed cheaply. 
This formulation has more variables and constraints, but it is negligible in big-O notation. 
This formulation has no interior solution, which is not problematic since we finally apply the proposed framework to the SDE model. 
Vector $w_{j}$ is the $j$\textsuperscript{th} column of matrix $W$ (or the $j$\textsuperscript{th} row of matrix $A$), and vector $v_{i}$ denotes the $i$\textsuperscript{th} column of matrix $V$. 
%

\begin{lemma}\label{lemma: orthogonality}
Vectors $w_{j}$ form a basis for the row space of $A$, and vectors $v_{i}$ form a basis for the null space of $A$. 
Consequently, for any $j\in\{1,\dots,m\} $ and any $i\in\{1,\dots,n-m\}$, we have $w_{j}^Tv_{i}=0$.
\end{lemma}
%

\begin{proof}
Since $A$ has full row rank or equivalently $A^T$ has full column rank, rows of $A$ (vectors $w_j$) form a basis for the range space of $A^T$ or row space of $A$. 
On the other hand, the matrix $V$ has full column rank because the vectors $v_i$ are linearly independent. 
Also, we have
$$W^TV=AV=\begin{bmatrix}
A_B&
A_N
\end{bmatrix}\begin{bmatrix}
A_B^{-1}A_N\\
-I
\end{bmatrix}=A_N-A_N=0.$$
We can conclude that the vectors $v_{i}$ form a basis for the null space of $A$ and $w_{j}^Tv_{i}=0$ for any $j\in\{1,\dots,m\} $ and any $i\in\{1,\dots,n-m\}$.
\end{proof}

%
Based on Lemma~\ref{lemma: orthogonality}, using $\lambda^T= (\lambda_1,\dots,\lambda_{n-m})$, we  reformulate~\eqref{eq:newton system} as 
\begin{subequations}\label{eq:orthogonal system}
\begin{alignat}{2}
    \Delta x&=\sum_{i=1}^{n-m}\lambda_i v_{i}=V\lambda\label{eq:orthogonal system delta x}\\
    \Delta s&=-\sum_{j=1}^{m} \Delta y_j w_{j}=-A^T \Delta y\label{eq:orthogonal system delta s}\\
    X\Delta s+S\Delta x&=\beta\mu e-Xs.\label{eq:orthogonal system xs}
\end{alignat} 
\end{subequations}
Substituting $\Delta x$ defined by equation~\eqref{eq:orthogonal system delta x}, and $\Delta s$ defined by equation~\eqref{eq:orthogonal system delta s} in equation~\eqref{eq:orthogonal system xs} results in
\begin{equation}\label{eq: OSS} \tag{OSS}
\begin{aligned}
    -X A^T\Delta y+SV\lambda &=\beta\mu e-Xs,
\end{aligned} 
\end{equation}
where vectors $\lambda$ and $\Delta y$ are unknown. 
One can rewrite \eqref{eq: OSS} as $M z=\sigma$ where
\begin{equation*}
    M=\begin{bmatrix}
    -XA^T& SV
\end{bmatrix},\qquad z=\begin{pmatrix}
\Delta y\\
\lambda
\end{pmatrix}, \qquad  \sigma=\beta\mu e-Xs.
\end{equation*}
We call this new system the ``Orthogonal Subspaces System" (OSS) which has $n$ equations, $n-m$ variables $\lambda_j$, and $m$ variables $\Delta y_i$. 
After solving~\eqref{eq: OSS}, $\Delta x$ and $\Delta s$ are calculated by \eqref{eq:orthogonal system delta x} and \eqref{eq:orthogonal system delta s}, respectively.
Based on Lemma~\ref{lemma: orthogonality}, we have $\Delta x^T \Delta s =0$.

\begin{lemma}\label{lemma: orthogonal to newton}
The linear systems~\eqref{eq:newton system} and~\eqref{eq: OSS} are equivalent.
\end{lemma}
%
The validity of Lemma~\ref{lemma: orthogonal to newton} can be verified by following the steps of deriving the~\eqref{eq: OSS}.
System~\eqref{eq:newton system} has a unique solution~\cite{Roos2005_Interior},
so an immediate consequence of Lemma~\ref{lemma: orthogonal to newton} is the following corollary. 
%

\begin{corollary}\label{lemma: unique solution}
If $(x, s, y)\in \mathcal{PD}^0$, then the system~\eqref{eq: OSS} has a unique solution.
\end{corollary}
%

Let $(\widetilde{\lambda},\widetilde{\Delta y})$ be an inexact solution of the system~\eqref{eq: OSS}. 
Then, we calculate approximate values $\widetilde{\Delta x}$ and $\widetilde{\Delta s}$ by using~\eqref{eq:orthogonal system delta x} and \eqref{eq:orthogonal system delta s}. 
The approximate solution $(\widetilde{\Delta x},\widetilde{\Delta s},\widetilde{\Delta y})$ satisfies
\begin{equation}\label{eq:inexact orthogonal system}
\begin{aligned}
    \widetilde{\Delta x}&=\sum_{j=1}^{n-m}\widetilde{\lambda_j} v_{j}=V\widetilde{\lambda}\\
    \widetilde{\Delta s}&=-\sum_{i=1}^{m} \widetilde{\Delta y_i} w_{i}=- W \widetilde{\Delta y},\\
    X\widetilde{\Delta s}+S\widetilde{\Delta x}&=\beta\mu e-Xs+r,
\end{aligned} 
\end{equation}
where $r$ is the residual in solving the \eqref{eq: OSS} inexactly. Let $(\lambda, \Delta y)$ represent the exact solution of~\eqref{eq: OSS}, then
$$r=SV(\widetilde{\lambda}-\lambda)-XA^T (\widetilde{\Delta y}-\Delta y).$$
It is important to emphasize that regardless of the error of the solution, we have $\widetilde{\Delta x}\in \text{Null}(A)$ and $\widetilde{\Delta s}\in \text{Row}(A)$. 
Thus, for any step length $\alpha\in (0,1]$, we have 
\begin{equation}\label{eq:feasibility of inexact solution}
\begin{aligned}
    A(x+\alpha \widetilde{\Delta x})&=b,\\
    A^T (y+\alpha \widetilde{\Delta y})+(s+\alpha \widetilde{\Delta s})&=c.
\end{aligned} 
\end{equation}
It implies the inexact Newton step calculated by solving \eqref{eq: OSS}, with appropriate step length, remains in the feasible region. 
This feature of the OSS enables us to develop an IF-IPM in the following section. 
%
%
\subsection{IF-IPM}
To develop a polynomially convergent IF-IPM, we enforce the following bound for the residual of inexact solution,
\begin{equation}\label{eq:residual bound}
    \|r^k\|\leq\eta \mu^k,
\end{equation}
where $\eta$ is an enforcing parameter with $0\leq\eta< 1$.
Let $\epsilon^k$ be the target error of the solution at iteration $k$, such that
$$
\Big\|\begin{pmatrix}\widetilde{\lambda}^k-\lambda^k, & \widetilde{\Delta y}^k-\Delta y^k\end{pmatrix}\Big\|_2\leq \epsilon^k.
$$
Then, we have 
$$\|r^k\|=\|\sigma^k- M^k \ztilde^k\|\leq \|M^k\|\|z^k-\tilde{z}^k\|\leq \|M^k\|\epsilon^k.$$
Thus, to satisfy~\eqref{eq:residual bound}, we need
$\epsilon^k\leq \eta \frac{\mu^k}{\|M^k\|}.$

Algorithm~\ref{alg: IF-IPM} is a short-step IF-IPM to solve LO problems using the general scheme of short-step feasible IPMs \cite{Roos2005_Interior}, in which iterates move in a small neighborhood of the central path while the complementarity gap is reduced by a fixed fraction.
%
\begin{algorithm}[] 
\caption{Short-step IF-IPM } 
\label{alg: IF-IPM}
\begin{algorithmic}[1]
 \STATE Choose $\zeta >0$, $\eta=0.1$, $\theta=0.2$ and $\beta=(1-\frac{0.11}{\sqrt{n}})$. 
 \STATE $k \gets 0$
 \STATE Choose initial feasible interior solution $(x^0, y^0, s^0)\in \mathcal{N}(\theta)$ 
\WHILE {$(x^k,y^k,s^k)\notin \mathcal{PD}_\zeta$}

\STATE $\mu^k \gets \frac{\big(x^k\big)^Ts^k}{n}$\label{alg-step:mu_k}
\STATE $\epsilon^k \gets \eta\frac{\mu^k}{\|M^k\|_2}$\label{alg-step:epsilon_k}
\STATE $(\lambda^k,\Delta y^k) \gets$ \textbf{solve}  (\ref{eq: OSS}) with error bound $\epsilon^k$  \label{alg-step:QLSA}
\STATE $\Delta x^k= V\lambda^k$ and $\Delta s^k=- A^T\Delta y^k$ \label{alg-step:x and s}
\STATE $(x^{k+1},y^{k+1},s^{k+1}) \gets (x^k,y^k,s^k)+(\Delta x^k,\Delta y^k,\Delta s^k)$\label{alg-step:update solution}
\STATE $k \gets k+1$
\ENDWHILE
\RETURN{$(x^k,y^k,s^k)$}
\end{algorithmic}
\end{algorithm}
%
In the next section, we prove the polynomial complexity of IF-IPM.
We also show that the proposed IF-IPM can attain the best iteration complexity $\Ocal(\sqrt{n}L)$ even with an inexact solution of the OSS system.

\subsection{Polynomial Convergence of IF-IPM} \label{sec: Convergence of IF-QIPM}
%
To prove the polynomial convergence of IF-IPM, in Theorem~\ref{theorem: iteration bound}, we show that $\mu^k$, which is a measure of the optimality gap, decreases linearly. 
To do so, Lemma~\ref{lemma: remaining in the neighbor} proves that the IF-IPM remains in the $\Ncal(\theta)$ neighborhood of the central path with a full step at each iteration. 
The main step in Theorem~\ref{theorem: iteration bound} is to show that the IF-IPM finds a $\zeta$-optimal solution after a polynomial number of iterations. 
Finally, we discuss the complexity of IF-IPM to find an exact solution. 
The first step is to demonstrate the correctness of Lemma~\ref{lemma: orthogonality and mu}.

\begin{lemma}\label{lemma: orthogonality and mu}
Let step $\Big(\widetilde{\Delta x}^k,\widetilde{\Delta y}^k, \widetilde{\Delta s}^k\Big)$ be obtained by~\eqref{eq: OSS} at the $k$\textsuperscript{th} iteration of the IF-IPM. 
Then
\begin{subequations}\label{eq: lemma of basics}
\begin{alignat}{2}
    \Big(x^k+\widetilde{\Delta x}{}^k\Big)^T \Big(s^k+\widetilde{\Delta s}{}^k\Big) &\leq \biggfl \beta+\frac{\eta}{\sqrt{n}} \biggfr \big(x^k\big)^Ts^k, \label{eq: compu}\\
    \Big(x^k+\widetilde{\Delta x}{}^k\Big)^T \Big(s^k+\widetilde{\Delta s}{}^k\Big) &\geq \biggfl \beta-\frac{\eta}{\sqrt{n}} \biggfr \big(x^k\big)^Ts^k. \label{eq: compl}
\end{alignat} 
\end{subequations}
\end{lemma}

\begin{proof}
To prove~\eqref{eq: compu}, we have
\begin{subequations}
\begin{alignat}{2}
  \Big(x^k+\widetilde{\Delta x}^k \Big)^T \Big(s^k+\widetilde{\Delta s}^k\Big) 
    &= \big(x^k\big)^T s^k + \big(x^k\big)^T \widetilde{\Delta s}^k + \big(s^k\big)^T \widetilde{\Delta x}^k +\big(\widetilde{\Delta x}^k\big)^T \widetilde{\Delta s}^k, \\
    &\leq \big(x^k\big)^T s^k + n\beta\mu^k -\big(x^k\big)^T s^k+\|r^k\|_1 + 0, \label{eq: l4c}\\
    &\leq n\beta\mu^k+\sqrt{n}\eta  \mu^k, \label{eq: l4d}\\
    &=  \biggfl\beta+\frac{\eta}{\sqrt{n}}\biggfr \big(x^k\big)^Ts^k\label{eq: l4e}.
\end{alignat}
\end{subequations}
Based on Lemma~\ref{lemma: orthogonal to newton}, we can use the last equation of~\eqref{eq:newton system} in line~\eqref{eq: l4c}. 
Inequality~\eqref{eq: l4d} follows from the residual bound~\eqref{eq:residual bound}, and~\eqref{eq: l4e} follows from the definition of $\mu^k$. 
Similarly, we can show that
\begin{equation*}
    \begin{aligned}
    (x^k+\widetilde{\Delta x}^k)^T(s^k+\widetilde{\Delta s}^k)&\geq \big(x^k\big)^T s^k + n\beta\mu^k -\big(x^k\big)^T s^k-\|r^k\|_1\\
    & \geq \biggfl\beta-\frac{\eta}{\sqrt{n}}\biggfr \big(x^k\big)^Ts^k.
\end{aligned}
\end{equation*}
The proof is complete.
\end{proof}

Lemma~\ref{lemma: remaining in the neighbor} proves that the iterates of the IF-IPM remain in the neighborhood of the central path. 
It follows from Lemma~\ref{lemma: orthogonality and mu}.

\begin{lemma} \label{lemma: remaining in the neighbor}
Let $\big(x^k,s^k,y^k\big)\in \mathcal{N}(\theta)$, then 
$\big(x^{k+1},s^{k+1},y^{k+1}\big)\in \mathcal{N}(\theta)$ for all $k \in \mathbb{N}$.
\end{lemma}

\begin{proof}
It is enough to show that 

\begin{subequations}
\begin{alignat}{2}
Ax^{k+1}&=b,\label{eq: l51a}\\
A^Ty^{k+1}+s^{k+1}&=c, \label{eq: l51b}\\ 
\big(x^{k+1}, \ s^{k+1} \big)&>0,\label{eq: l51c}\\
\big\|X^{k+1}S^{k+1}e-\mu^{k+1} e \big\|_2&\leq \theta \mu^{k+1}, \qquad  \forall i\in\{1,\dots,n\}.\label{eq: l51d}
\end{alignat}
\end{subequations}
We can derive equalities~\eqref{eq: l51a} and~\eqref{eq: l51b} from equation (\ref{eq:feasibility of inexact solution}). To prove~\eqref{eq: l51d}, first we show that 
$\big\|\Delta X^k \Delta S^k e\big\|\leq \frac{\theta^2 +n(1-\beta)^2+\eta^2}{2^{\frac{3}{2}}(1-\theta)}\mu^k$. 
Let $D= \big(X^k \big)^{\frac{1}{2}} \big(S^k \big)^{-\frac{1}{2}}$, then we have
\begin{subequations}
\begin{alignat}{2}
    \big\|\Delta X^k \Delta S^k e\big\| &=\big\|\big(D^{-1}\Delta X^k\big) \big(D\Delta S^k\big)e\big\|\label{eq: l52a}\\
    &\leq2^{-\frac{3}{2}} \big\|D^{-1}\Delta x^k+D\Delta s^k \big\|^2\label{eq: l52b}\\
    &=2^{-\frac{3}{2}} \big\|(X^kS^k)^{-\frac{1}{2}}(S^k\Delta x^k+X^k\Delta s^k)\big\|^2\label{eq: l52c}\\
    &=2^{-\frac{3}{2}} \big\|(X^kS^k)^{-\frac{1}{2}}(\beta \mu^ke-X^kS^ke+r^k) \big\|^2\label{eq: l52d}\\
    &=\sum_{i=1}^{n}\frac{(\beta \mu^k-x_i^ks_i^k+r_i^k)^2}{2^{\frac{3}{2}}x_i^ks_i^k}\label{eq: l52e}\\
    &\leq\frac{\|\beta \mu^ke-X^kS^ke+r^k\|^2}{2^\frac{3}{2} \min_{i}x_i^ks_i^k}\label{eq: l52f}\\
    &\leq\frac{\|\beta \mu^ke-X^kS^ke\|^2+\|r^k\|^2}{2^\frac{3}{2}(1-\theta)\mu^k}\label{eq: l52g}\\
    &\leq \frac{\|(X^kS^ke- \mu^ke)+(1-\beta)\mu^ke\|^2+(\eta\mu^k)^2}{2^\frac{3}{2} (1-\theta)\mu^k}\label{eq: l52h}\\
    &\resizebox{0.7\hsize}{!}{$\leq\frac{\|(X^kS^ke- \mu^ke)\|^2 +2(1-\beta)\mu^ke^T(X^kS^ke- \mu^ke)+n((1-\beta)\mu^k)^2+(\eta\mu^k)^2}{2^\frac{3}{2}(1-\theta)\mu^k}$}\label{eq: l52i}\\
    &\leq \frac{(\theta\mu^k)^2 +n((1-\beta)\mu^k)^2+(\eta\mu^k)^2}{2^\frac{3}{2} (1-\theta)\mu^k}\label{eq: l52j}\\
    &\leq\frac{\theta^2 +n(1-\beta)^2+\eta^2}{2^{\frac{3}{2}}(1-\theta)}\mu^k.\label{eq: l52k}
\end{alignat}
\end{subequations}

Equation~\eqref{eq: l52b} follows from Lemma 5.3 of~\cite{Wright1997_Primal},~\eqref{eq: l52e} from equation~\eqref{eq:inexact orthogonal system},~\eqref{eq: l52g} from $\min_i x_i^ks_i^k\geq (1-\theta) \mu^k$,~\eqref{eq: l52h} from the residual bound~\eqref{eq:residual bound},  and~\eqref{eq: l52j} from the definition of the neighborhood. 
We now prove inequality~\eqref{eq: l51d} as follows
\begin{subequations}
\begin{alignat}{2}
    \|X^{k+1}S^{k+1}e-\mu^{k+1} e\|_2 &=\sqrt{\sum_{i=1}^{n}\big( (x^{k}_i+\Delta x^k_i) (s^{k}_i+\Delta s^k_i)-\mu^{k+1}\big)^2\label{eq: l53a}}\\
    &= \sqrt{\sum_{i=1}^{n} \big( \beta \mu^k +\Delta x^k_i\Delta s^k_i+r_i^k-\mu^{k+1} \big)^2 \label{eq: l53b}}\\
    &\leq \big\|\Delta X^k\Delta S^ke \big\| + \sqrt{n} \big|\beta\mu^k-\mu^{k+1} \big|+  \big\|r^k \big\|\label{eq: l53c}\\
    &\leq \frac{\theta^2 +n(1-\beta)^2+\eta^2}{2^{\frac{3}{2}}(1-\theta)}\mu^k+ \sqrt{n} \big|\beta\mu^k-\mu^{k+1} \big|+\eta \mu^k\label{eq: l53d}\\
    &\leq \frac{\theta^2 +n(1-\beta)^2+\eta^2}{2^{\frac{3}{2}}(1-\theta)}\mu^k+ 2\eta \mu^k\\
    &\leq \bigg( \frac{\theta^2 +n(1-\beta)^2+\eta^2}{2^{\frac{3}{2}}(1-\theta)}+ 2\eta\bigg) \frac{\mu^{k+1}}{\beta-\frac{\eta}{\sqrt{n}}}\label{eq: l53f}\\
    &\leq\theta \mu^{k+1}.\label{eq: l53g}
\end{alignat}
\end{subequations}
Equation~\eqref{eq: l53b} follows from system~\eqref{eq:inexact orthogonal system},~\eqref{eq: l53c} form the triangular inequality, and~\eqref{eq: l53d} form Lemma~\ref{lemma: orthogonality and mu}. 
One can easily verify that inequality~\eqref{eq: l53g} holds for $(\eta, \theta, \beta)=(0.1,0.2,1-\frac{0.11}{\sqrt{n}})$.
For $0\leq\alpha\leq 1$, let $x_i^k(\alpha)=x_i^k+\alpha(\Delta x_i^k)$ and $s_i^k(\alpha)=s_i^k+\alpha(\Delta s_i^k)$. 
We have $(x_i^k(0), s_i^k(0))>0$ for all $i\in\{1,\dots,n\}$. 
Based on the previous step and Lemma~\ref{lemma: orthogonality and mu}, we have 
$$x_i^k(\alpha)s_i^k(\alpha)\geq (1-\theta)\mu^k(\alpha)\geq(1-\theta)\left(\beta-\frac{\eta}{\sqrt{n}}\right)\mu^k>0.$$
We have $(x^{k+1}_i, s^{k+1}_i)>0$ because we can not have $x_i^k(\alpha)=0$ and $s_i^k(\alpha)=0$ for any $i\in\{1,\dots,n\}$ and $\alpha\in[0,1]$. 
Thus, inequality \eqref{eq: l51c} is proved, and the proof is complete.
\end{proof}

Based on Lemma~\ref{lemma: remaining in the neighbor}, IF-IPM remains in the neighborhood of the central path, and it converges to the optimal solution if $\mu^k$ converges to zero. In Theorem~\ref{theorem: iteration bound}, we prove that the algorithm reaches the $\zeta$-optimal solution after a polynomial time.

\begin{theorem}\label{theorem: iteration bound}
The sequence $\mu^k$ converges to zero linearly, and we have $\mu^k\leq \zeta$ after
$\Ocal(\sqrt{n}\log(\frac{{\mu}_0}{\zeta}))$ 
iterations.
\end{theorem}
%

\begin{proof}
By Lemma~\ref{lemma: orthogonality and mu}, we have
\begin{equation*}
    \mu^{k+1}\leq\left(\beta+\frac{\eta}{\sqrt{n}}\right)\mu^k=  \left(1-\frac{0.01}{\sqrt{n}}\right)\mu^{k}\leq\left(1-\frac{0.01}{\sqrt{n}}\right)^k\mu^{0}.
\end{equation*}
Since $\mu^k$ is bounded below by zero, and it is monotonically decreasing, it converges linearly to zero.
Since the IF-IPM stops when $\mu^{k}\leq \zeta$, then we have
\begin{align*}
    \left(1-\frac{0.01}{\sqrt{n}}\right)^k&\leq \frac{\zeta}{\mu^{0}} \Rightarrow
    \frac{\sqrt{n}}{0.01}\log \left(\frac{\mu^{0}}{\zeta}\right)\leq k.
\end{align*}
Thus, IF-IPM has $\Ocal(\sqrt{n}\log(\frac{{\mu}_0}{\zeta}))$ iteration complexity.
\end{proof}

%
As the proof shows, the IF-IPM has polynomial complexity for any values of the parameters satisfying conditions~\eqref{eq: param con1} and~\eqref{eq: param con2}:
\begin{align}
    \left(\beta+\frac{\eta}{\sqrt{n}}\right)&\leq\left(1-\frac{0.01}{\sqrt{n}}\right),\label{eq: param con1}\\
    \left(\beta-\frac{\eta}{\sqrt{n}}\right)&\geq0,\label{eq: param con2}\\
    \left(\frac{\theta^2 -n(1-\beta)^2+\eta^2}{2^{3/2}(1-\theta)}+ 2\eta\right) &\leq \theta\left(\beta-\frac{\eta}{\sqrt{n}}\right).\label{eq: param con3}
\end{align}
It is not hard to check that $\theta=0.2$ and $\eta=0.1$ satisfy these conditions.
Let  $L$ be the binary length of input data~\cite{Wright1997_Primal} defined as 
$$L=mn+m+n+\sum_{i,j}\lceil\log(|a_{ij}|+1)\rceil+\sum_{i}\lceil\log(|c_{i}|+1)\rceil+\sum_{j}\lceil\log(|b_{j}|+1)\rceil.$$

An exact solution can be calculated by rounding~\cite{Wright1997_Primal} if  $\mu^k\leq2^{\Ocal(L)}.$
Thus, the upper bound for the number of iterations of our IF-IPM to find an exact optimal solution is $\Ocal(\sqrt{n}L)$ (for more details, see Chapter 3 of~\cite{Wright1997_Primal}). In the next section, we analyze the OSS more and compare it to other Newton systems.

\section{Analyzing the Orthogonal Subspaces System}\label{sec: anal system} 
In this section, first, we analyze the condition number of the matrix of  the~\eqref{eq: OSS}. Then, the new system will be compared to other systems.

\subsection{The Condition Number of \texorpdfstring{\pmb{$M$}}{M}}
By the definition of the neighborhood of the central path, for each pair of primal and dual variables $(x_i^k, s_i^k)$, the following relationship holds
\begin{equation*}
    |x_i^k s_i^k-\mu^k|\leq \|X^k S^k e - \mu^k e\|_2\leq \theta \mu^k\quad \Rightarrow \quad  (1-\theta)\mu^k\leq x_i^ks_i^k\leq (1+\theta)\mu^k.
\end{equation*}
So we can rewrite $X^k S^k$ as
$X^kS^k = \mu^kI + \theta \mu^k \mathcal{L}^k,$ where $\mathcal{L}^k$ is a diagonal matrix with both $I-\mathcal{L}^k$ and $I+\mathcal{L}^k$ positive semi-definite. 
Recall that $AV=0$,
then
\begin{align*}
    (M^k)^T M^k 
    &= \begin{bmatrix}
    A(X^k)^2A^T & -\theta \mu^k A\mathcal{L}^kV\\-\theta \mu^k V^T\mathcal{L}^kA^T & V^T(S^k)^2V
    \end{bmatrix}\\
    &= \begin{bmatrix}
    A &0\\
    0 & V^T
    \end{bmatrix}\begin{bmatrix}
    (X^k)^2 & -\theta\mu^k \mathcal{L}^k\\
    -\theta\mu^k \mathcal{L}^k & (S^k)^2
    \end{bmatrix}\begin{bmatrix}
    A^T &0 \\
    0& V
    \end{bmatrix}.
\end{align*}
It worth noting that the decomposition above is not unique since $AV=0$. There might be a better decomposition choice that might give a tighter bound for the condition number. In this work, for its simplicity, we use this choice.
With the submultiplicativity of the spectral norm, we can easily have the following lemma.
\begin{lemma}\label{lemma: eigen of matrix product}
For any full row rank matrix $Q\in \Rmbb^{m\times n}$ and any symmetric positive definite matrix $\Psi\in \Rmbb^{n\times n}$, their condition number satisfies
$$
\kappa(Q\Psi Q^T) = \Ocal\left(\kappa(\Psi)\right).$$
\end{lemma}

To apply Lemma \ref{lemma: eigen of matrix product} to $(M^k)^T M^k$, we need to show that the middle matrix in the decomposition is symmetric positive definite. Clearly, the matrix is symmetric. We only need to show that all of its eigenvalues are positive. 
Take the following notation,
$$U^k = \begin{bmatrix}
    (X^k)^2 & -\theta\mu^k \mathcal{L}^k\\
    -\theta\mu^k \mathcal{L}^k & (S^k)^2
    \end{bmatrix}.$$
Notice that the four blocks of $U^k$ are all square and diagonal, so $U^k$ is square and symmetric. 
For the sake of simplicity, in the remainder of this section, we omit the superscript $k$.
Let equate the characteristic polynomial of $U^k$ to zero. We can get all the eigenvalues of $U^k$ as
\begin{equation*}
    \frac{1}{2}\Bigg((x_i^2 + s_i^2) \pm \sqrt{(x_i^2 + s_i^2)^2 - 4 x_i^2s_i^2 + 4\theta^2\mu^2\ell_i^2}   \Bigg)
\end{equation*}
for $i=1,\dots,n$, where $\ell_i$ is the $i^{\rm th}$ diagonal element of $\mathcal{L}$.
If the smallest eigenvalue is positive, then $U^k$ is symmetric positive definite. The smallest eigenvalue denoted as $\iota_{\min}$, can be bounded as follows
\begin{align*}
    \iota_{\min} &= \min_i \frac{1}{2}\Bigg((x_i^2 + s_i^2) - \sqrt{(x_i^2 + s_i^2)^2 - 4 x_i^2s_i^2 + 4\theta^2\mu^2\ell_i^2}   \Bigg)\\
        &=\min_i \frac{(x_i^2 + s_i^2)}{2}\Bigg(1 - \sqrt{1 + \frac{-4 x_i^2s_i^2 + 4\theta^2\mu^2\ell_i^2}{(x_i^2 + s_i^2)^2}}   \Bigg)\\
        &=\min_i \frac{(x_i^2 + s_i^2)}{2}\Bigg(1 - \sqrt{1 + \frac{4(- x_is_i + \theta\mu\ell_i)( x_is_i + \theta\mu\ell_i)}{(x_i^2 + s_i^2)^2}}   \Bigg).
\end{align*}
Since $U^k$ is a real symmetric matrix, its eigenvalues are real; therefore, we can use $(x_i^2 + s_i^2)^2 - 4 x_i^2s_i^2 + 4\theta^2\mu^2\ell_i^2 \geq 0$ to show
\begin{equation*}
\begin{aligned}
    \frac{4 x_i^2s_i^2 - 4\theta^2\mu^2\ell_i^2}{(x_i^2 + s_i^2)^2} \leq 1 \Rightarrow
    \frac{4 (x_i s_i  -\theta\mu\ell_i)(x_i s_i  +\theta\mu\ell_i)}{(x_i^2 + s_i^2)^2} \leq 1. 
\end{aligned}
\end{equation*}
Recall the definition of $\mathcal{L}^k$ in this section, it follows that

\begin{align*}
       1\geq \frac{4 (x_i s_i  -\theta\mu\ell_i)(x_i s_i  +\theta\mu\ell_i)}{(x_i^2 + s_i^2)^2} = \frac{4 (\mu + \theta \mu \ell_i  -\theta\mu\ell_i)(\mu + \theta \mu \ell_i +\theta\mu\ell_i)}{(x_i^2 + s_i^2)^2}
       = \frac{4\mu^2(1 + 2\theta\ell_i)}{(x_i^2 + s_i^2)^2}.
\end{align*}
It follows that
\begin{align*}
    \iota_{\min} &= \min_i \frac{(x_i^2 + s_i^2)}{2}\Bigg(1 - \sqrt{1 + \frac{4\mu^2(1 + 2\theta\ell_i)}{(x_i^2 + s_i^2)^2}}   \Bigg)\\
    &\geq\min_i \frac{(x_i^2 + s_i^2)}{2}\Bigg(1 - \left(1 - \frac{1}{2}\frac{4\mu^2(1 + 2\theta\ell_i)}{(x_i^2 + s_i^2)^2}\right)   \Bigg)\\
       &=\min_i  \frac{\mu^2(1 + 2\theta\ell_i)}{(x_i^2 + s_i^2)},
\end{align*}
where the inequality holds because $\sqrt{1-t}\leq 1-\frac{1}{2}t$ for any $t\leq 1$.
When $\theta\in [0,\frac{1}{4}]$ and $\|x\|,\|s\|\leq \omega$, it follows that
$\iota_{\min}\geq \frac{\mu^2}{4\omega^2}>0.$
Analogously, we have
\begin{align*}
    \iota_{\max} &= \max_i \frac{1}{2}\Bigg((x_i^2 + s_i^2) + \sqrt{(x_i^2 + s_i^2)^2 - 4 x_i^2s_i^2 + 4\theta^2\mu^2e_i^2}   \Bigg)\\
    &\leq \max_i \frac{1}{2}\Bigg((x_i^2 + s_i^2) + \sqrt{(x_i^2 + s_i^2)^2}   \Bigg)\leq 2\omega^2.\label{eq: eigmax}
\end{align*}
So the condition number of $U^k$ is bounded by $\kappa(U^k) \leq \frac{8\omega^4}{(\mu^k)^2}$
and the condition number of $M^k$ satisfies
\begin{align}
    \kappa(M^k)=\Ocal\left( \frac{\omega^2}{\mu^k} \kappa_Q\right),\label{eq: cond_oss dependence}
\end{align}
where $\kappa_Q$ is the condition number of constant matrix $Q=\begin{bmatrix}
    A &0\\
    0 & V^T
    \end{bmatrix}$.

\subsection{Comparing Different Systems}
To compute the Newton step, one can solve the Full Newton System (FNS), whose coefficient matrix is
\begin{equation}\label{eq:full newton system}
\begin{bmatrix}
0&A&0\\
A^T&0&I\\
0&S^k&X^k
\end{bmatrix}.
\end{equation}
The FNS can be simplified to the Augmented System (AS), which has the coefficient matrix
\begin{equation}\label{eq:augmented system}
\begin{bmatrix}
0&A\\
A^T&-(X^k)^{-1}S^k
\end{bmatrix}.
\end{equation}
We can simplify the AS to get the Normal Equation System (NES) with coefficient matrix
\begin{equation}\label{eq:normal equation}
    AX^k (S^k)^{-1}A^T.
\end{equation}
Many of the implementations of IPMs use the NES since it has a small positive definite matrix and can be solved efficiently by Cholesky factorization~\cite{Wright1997_Primal}. 
Table~\ref{tab: systems} compares the properties of the different systems.
Although the NES is smaller than other systems, the NES is typically much denser.
The coefficient matrix of the NES is dense if matrix A has dense columns. 
We can use the Sherman-Morrison-Woodbury formula \cite{horn2012matrix} to solve the NES with sparse matrix efficiency if matrix $A$ has only a few dense columns \cite{lowrank}. 
The OSS has better sparsity than the NES since the sparsity  of its coefficient matrix $\begin{bmatrix}
    -XA^T& SV
\end{bmatrix}$ is determined by the sparsity of $A$ and $V$.
By sparse $A$, and appropriate choice of the basis, the OSS can be much sparser than the NES.
If we solve FNS, AS, or NES inexactly, the potential infeasibility will increase the complexity of IPMs. 
Thus, the~\ref{eq: OSS} is more adaptable with inexact solvers such as QLSAs and the classical iterative method. 
Another reason is that the condition number of OSS has the square root of the rate of the growth than other systems. 
Most of the inexact solvers, both classical and quantum, are sensitive to the condition number. 
Despite its high adaptability for inexact solvers,  the proposed OSS is larger than the NES but smaller than the FNS and AS, and it is nonsingular but not positive-definite. 
Thus, we can not solve it by Cholesky factorization. We can use LU factorization instead.
In Section~\ref{alg: IF-QIPM}, we discuss how we can use QLSA efficiently in IF-QIPMs and how much the OSS is more adaptable to QLSAs than the other systems. 
%
%
\begin{table}[]
    \centering
    \resizebox{\columnwidth}{!}{
\begin{tabular}{ |c|c|c|c|c| } 
 \hline
 System & Size of system & Symmetric & Positive Definite&  Rate of the Condition Number Growth\\ 
 \hline
FNS  & $2n+m$& \xmark& \xmark&  $\Ocal\big(\frac{1}{\mu^2}\big)$\\
AS & $n+m$& \cmark& \xmark&  $\Ocal\big(\frac{1}{\mu^2}\big)$\\
NES  & $m$ & \cmark& \cmark& $\Ocal\big(\frac{1}{\mu^2}\big)$\\
\ref{eq: OSS}  & $n$ & \xmark& \xmark&  $\Ocal\big(\frac{1}{\mu}\big)$\\
\hline
\end{tabular}
}
\caption{Characteristics of the Coefficient Matrices of Different Newton Systems}
    \label{tab: systems}
\end{table}

\section{IF-QIPM with QLSAs} \label{sec: IF-QIPM}
%
We employ a similar approach to~\cite{mohammadisiahroudi2022efficient} to couple the QLSA with the proposed IF-IPM. 
The HHL algorithm proposed by~\cite{Harrow2009_quantum} was the first QLSA for solving a quantum linear system with $p$-by-$p$ Hermitian matrix in $\tilde{\Ocal}_{p}\big(\frac{d^2\kappa^2}{\epsilon}\big)$ time complexity.
Here, $\epsilon$ is the target error, $\kappa$ is the condition number of the coefficient matrix, and $d$ is the maximum number of nonzero entries in every row or column. 
After the HHL method, several QLSAs were proposed with better time complexity than the HHL method. 
%
%
Wossnig et al.~\cite{QLSA} proposed a QLSA algorithm independent of sparsity with $\tilde{\Ocal}_{p}\Big(\|M\|_F\frac{\kappa}{\epsilon}\Big)$ complexity. 
Childs et al.~\cite{Childs} developed a  QLSA with exponentially better dependence on error with $\tilde{\Ocal}_{p,\kappa,\frac{1}{\epsilon}}(d\kappa)$ complexity. In another direction, QLSAs using Block Encoding have $\Ocal_{p,\frac{1}{\epsilon}}(\|M\|_F\kappa)$ complexity~\cite{Blockencoding}. To encode the linear system in a quantum setting and solve it by QLSA, we need a procedure discussed in~\cite{mohammadisiahroudi2022efficient}. To solve the~\ref{eq: OSS}, we build system $M'^k z'^k=\sigma'^k$ where
\begin{equation} \label{eq: transformed system}
      M'^k = \frac{1}{\|M^k\|}\bbmatrix 0 & M^k\\ {M^k}^T & 0\ebmatrix,\ z'^k = \bpmatrix 0 \\ z^k\epmatrix\text{, and }
    \sigma'^k = \frac{1}{\|M^k\|}\bpmatrix \sigma^k \\ 0\epmatrix.  
\end{equation}
The new system can be implemented in a quantum setting and solved by QLSA since $M'^k$ is a Hermitian matrix and $\|M'^k\|=1$. To extract a classical solution, we use the Quantum Tomography Algorithm (QTA) by~\cite{van2021tomography}. Theorem~\ref{theorem: QLSA solution} shows how we can adapt QLSA by~\cite{Blockencoding} to solve the OSS system.

\begin{theorem} \label{theorem: QLSA solution}
Given the linear system (\ref{eq: OSS}), QLSA and QTA provide the solution $(\widetilde{\lambda^k},\widetilde{\Delta y^k})$ with residual $r^k$, where $\|r^k\|\leq \eta \mu^k$, in at most $$\tilde{\Ocal}_{n,\kappa_Q,\omega,\frac{1}{\mu^k}}\left(n\frac{\kappa_Q^2\|Q\|\omega^5}{(\mu^k)^2}\right)$$
time complexity.
\end{theorem}

\begin{proof}
We can derive the transformed system~\eqref{eq: transformed system} from~\eqref{eq: OSS}. To have $\|r^k\|\leq \eta \mu^k$, the error of the linear system $\epsilon_{LS}$ must be less than $\frac{\eta \mu^k}{\|M^k\|}$. Since scaling affects the error of QLSA, we need to find an appropriate bound for QLSA, and QTA~\cite{mohammadisiahroudi2022efficient}. Based on the analysis done in the second section of \cite{mohammadisiahroudi2022efficient},  the complexity of QLSA of \cite{Blockencoding} is $\Ocal(\kappa_{M^k}\|M^k\|_F\text{polylog}(\frac{n\kappa_{M^k}\|\sigma^k\|}{\mu^k}))$, and the complexity of QTA of \cite{van2021tomography} is $\Ocal(\frac{n\kappa_{M^k}\|\sigma^k\|}{\mu^k})$.
Further, based on the definition of the neighborhood of the central path, we have
\begin{align*}
    \frac{\|\sigma^k\|}{\mu^k}&=\frac{\|\beta\mu^k e-X^ks^k\|}{\mu^k}\leq\frac{\|\beta \mu^k e-\mu^k e\|+\|\mu^k e-X^ks^k\|}{\mu^k}\leq (1+\theta-\beta).
\end{align*}
As we can see, the error bound is fixed for all iterations and shows high adaptability of this system for QLSA and QTA. 
Since $\kappa_{M^k}=\Ocal(\frac{\omega^2\kappa_Q}{\mu^k})$ and $\|M^k\|_F=\Ocal(\sqrt{n}\omega\|Q\|)$, the QLSA by~\cite{Blockencoding} can find such a solution with $\tilde{\Ocal}_{n,\kappa_Q,\omega,\frac{1}{\mu^k}}(\frac{\sqrt{n}\kappa_Q\|Q\|\omega^3}{\mu^k})$
time complexity. The time complexity of QTA by~\cite{van2021tomography} is $\Ocal(\frac{n\kappa_Q\omega^2}{\mu^k})$. So, the total complexity is 
$$\tilde{\Ocal}_{n,\kappa_Q,\omega,\frac{1}{\mu^k}}\left(n^{1.5}\frac{\kappa_Q^2\|Q\|\omega^5}{(\mu^k)^2}\right).$$
The proof is complete.
\end{proof}

There are few studies investigating Quantum Interior Point Methods (QIPMs) for LO problems. First, ~\cite{kerenidisParkas2020_quantum} used Block Encoding and QRAM for finding a $\zeta$-optimal solution with $\tilde{\Ocal}(\frac{n^{2}}{\epsilon^{2}}\bar{\kappa}^3 \log(\frac{1}{\zeta}))$ complexity, where $\bar{\kappa}$ is an upper bound on the condition number of the Newton systems. Casares and Martin-Delgado~\cite{Casares2020_quantum} used QLSA and developed a Predictor-correcter QIPM with $\tilde{\Ocal}(L\sqrt{n}(n+m)\|\bar{M}\|\frac{\bar{\kappa}^2}{\epsilon^{-2}})$ complexity. Both papers used exact IPMs, which are not valid when a QLSA is used to solve the Newton systems. Augustino et al. \cite{augustino2021quantum} proposed two type of convergent QIPMs which addressed the issues of previous QIPMs for SDO. However, in all the proposed QIPMs, $\epsilon$ and $\bar{\kappa}$ increases exponentially, leading to exponential time complexities. To address this problem,~\cite{mohammadisiahroudi2022efficient} developed an II-QIPM using QLSA efficiently by using an iterative refinement method. The complexity of their IR-II-QIPM is  \begin{equation*}
\tilde{\Ocal}_{n,\phi,\kappa_{\Ahat}}\left(n^4L\phi\kappa_{\Ahat}^4\right),
\end{equation*}
where $\hat{A}$ and $\hat{b}$ are preprocessed $A$ and $b$, and $\phi=\omega^{19}(\|\Ahat\|+\|\bhat\|)$. To improve this time complexity, we can use Algorithm~\ref{alg: IF-QIPM}, which is a short-step IF-QIPM for solving LO problems using QLSA and QTA to solve system~\eqref{eq: OSS}. Theorem~\ref{theorem: IFQIPM iter} and Corollary~\ref{cor: IFQIPM time} show the iteration and total time complexities of the proposed IF-QIPM, respectively. 
%

\begin{algorithm}
\caption{IF-QIPM using QLSA} \label{alg: IF-QIPM}
\begin{algorithmic}[1]
 \STATE Choose $\zeta >0$, $\eta=0.1$, $\theta=0.3$ and $\beta=(1-\frac{0.11}{\sqrt{n}})$. 
 \STATE $k \gets 0$
 \STATE Choose initial feasible interior solution $(x^0, y^0, s^0)\in \mathcal{N}(\theta)$ 
\WHILE {$(x^k,y^k,s^k)\notin \mathcal{PD}_\zeta$}
\STATE $\mu^k \gets \frac{\big(x^k\big)^Ts^k}{n}$
\STATE $(M^k,\sigma^k) \gets$\textbf{build} (\ref{eq: OSS}) 
\STATE $(\lambda^k,\Delta y^k) \gets$ \textbf{solve} the (\ref{eq: OSS}) using QLSA and QTA 
\STATE $\Delta x^k= V\lambda^k$ and $\Delta s^k=- A^T\Delta y^k$ 
\STATE $(x^{k+1},y^{k+1},s^{k+1}) \gets (x^k,y^k,s^k)+(\Delta x^k,\Delta y^k,\Delta s^k)$
\STATE $k \gets k+1$
\ENDWHILE
\RETURN{$(x^k,y^k,s^k)$}
\end{algorithmic}
\end{algorithm}

\begin{theorem}\label{theorem: IFQIPM iter}
The IF-QIPM presented in Algorithm~\ref{alg: IF-QIPM} produces a $\zeta$-optimal solution after $\Ocal(\sqrt{n}\log(\frac{\mu^0}{\zeta}))$ iterations.
\end{theorem}
%
%
The proof of Theorem~\ref{theorem: IFQIPM iter} is analogous to Theorem~\ref{theorem: iteration bound}. 
\begin{corollary}\label{cor: IFQIPM time}
The detailed time complexity of the IF-QIPM presented in  Algorithm~\ref{alg: IF-QIPM} is
$$
\tilde{\Ocal}_{n,\kappa_Q,\omega,\frac{1}{\mu^k}}\left(\sqrt{n}\log \left(\frac{\mu^0}{\zeta}\right)\left(n^2+ n^{1.5}\frac{\kappa_Q^2\|Q\|\omega^5}{(\zeta)^2}\right)\right)
.
$$
\end{corollary}
\begin{proof}
    To reach a $\zeta$-optimal solution, one needs $\mu\leq \zeta$ in Theorem~\ref{theorem: QLSA solution}. 
\end{proof}
The time complexity depends on $\frac{1}{\zeta}$, which leads to exponential time for finding an exact solution. 
Section~\ref{sec: IR} uses an iterative refinement method to address this issue.
\section{IF-IPM using CGM}\label{sec: IF-IPM CGM}

In the proposed IF-QIPM, we solve the OSS system with QLSA+QTA to compute the Newton step. 
Newton steps can also be calculated by classical Conjugate Gradient methods.
We show in this section that the IF-IPM using CGM can lead to similar complexity to the one of IF-QIPM. 
\subsection{Calculating Newton Step by CGM}
A basic approach for solving the OSS system is Gaussian elimination, or LU factorization,  with $\Ocal(n^3)$ arithmetic operations. 
To reduce the cost of solving the OSS system, the best iterative method is the GMRES algorithm, which also has $\Ocal(n^3)$ worst-case complexity \cite{saad2003iterative}. 
For problems in the form of $E^TE z = E^T \psi$, known as normal equations, one can use a version of CGMs with complexity $\Ocal(nd\kappa_{E}\log(1/\epsilon))$, where $\kappa_E$ is the condition number of matrix $E$ \cite{saad2003iterative}. 
For a linear system in general form with a non-PSD non-symmetric coefficient matrix, such as the OSS system, one can use the reformulation $M^TM z =M^T \sigma$ and use a CGM to solve it. 
Although CGMs for this reformulation have better worst-case complexity than GMRES for the original system, practically GMRES has better performance, especially for large sparse systems with large condition number \cite{saad2003iterative}.
At each iteration of an IF-IPM, Algorithm~\ref{alg: IF-IPM}, we need to solve $M z= \sigma$ such that $\|\sigma-M\tilde{z}\|\leq \eta \mu$, where
\begin{equation*}
    M=\begin{bmatrix}
    -XA^T& SV
\end{bmatrix},\qquad z=\begin{pmatrix}
\Delta y\\
\lambda
\end{pmatrix}, \qquad  \sigma=\beta\mu e-Xs.
\end{equation*}
Here, we use a CGM (Algorithm 8.5 of \cite{saad2003iterative}) as specified in Algorithm~\ref{alg: CGNE}.
\begin{algorithm}[ht]
	\caption{CGM} \label{alg: CGNE}
	\begin{algorithmic}[1]
		\REQUIRE $\big(M\in \mathbb{R}^{n\times n}, \sigma \in \mathbb{R}^{n}\big)$ 
		\STATE  $k \gets 0$
		\STATE $r^{0} \gets M^T\sigma - M^T(M z^0)$ and $p^{0} \gets r^{0}$ 
		\WHILE{$\|r^k\|>\epsilon$}
		\STATE $w^k\gets M p^k$
		\STATE $\alpha^k\gets \frac{\|r^k\|^2}{\|w^k\|^2}$
		\STATE $z^{k+1}\gets z^{k}+\alpha^k z^k$ 
            \STATE $r^{k+1}\gets M^T r^{k}-\alpha^k  M^T w^k$ 
		\STATE $\beta^k\gets \frac{\|r^{k+1}\|^2}{\|r^k\|^2}$
		\STATE $p^{k+1} \gets r^k+\beta^k p^k$
		
		\STATE $k \gets k+1$
		\ENDWHILE
	\end{algorithmic}
\end{algorithm}
As we can see, there is no matrix-matrix product in this CGM. The following theorem presents the complexity of calculating the Newton step.
\begin{theorem} \label{theorem: CGNE complexity}
The computational complexity of Algorithm~\ref{alg: CGNE}, to find a solution $\tilde{z}$ such that $\|\sigma-M \tilde{z}\|\leq \eta \mu$, is  $\Ocal(n^2\kappa_{M}\log(\frac{\|r^0\|}{\mu}))$.
\end{theorem}
\begin{proof}
Similar to the proof of Theorem 6.29 of \cite{saad2003iterative}.
\end{proof}
\subsection{Total Complexity of IF-IPM using CGM}
Based on Theorem~\ref{theorem: iteration bound}, the iteration complexity is independent of how we calculate the Newton step. The iteration complexity of IF-IPM method as presented in Algorithm~\ref{alg: IF-IPM} is $\Ocal(\sqrt{n}\log(\frac{\mu^0}{\zeta}))$ when the residual in each iteration is bounded by $\eta \mu$. Furthermore, by this approach, we avoid matrix-matrix products. The following Theorem presents the total complexity of IF-IPM using CGM.
\begin{theorem}
Using IF-IPM as presented in Algorithm \ref{alg: IF-IPM} with CGM, a $\zeta$-optimal solution for an LO problem is attained using at most $\tilde{\Ocal}_{\mu^0/\zeta}(n^{2.5}\frac{\kappa_{Q}\omega^2}{\zeta})$ arithmetic operations.
\end{theorem}
\begin{proof}
Based on Theorem~\ref{theorem: CGNE complexity}, the cost of calculating the Newton direction at each iteration is $\Ocal(n^2\kappa_{M}\log(\frac{\|r^0\|}{\mu}))$, if CGM starts with $z^0=0$, then $r^0=\sigma$. Based on the analysis in the proof of Theorem~\ref{theorem: QLSA solution}, we have $\kappa_{M}=\Ocal(\frac{\kappa_{Q}\omega^2}{\zeta})$, and $\log(\frac{\|r^0\|}{\mu})=\log(\frac{\|\sigma\|}{\mu})=\Ocal(1)$. Thus, the complexity of CGM is bounded by $\Ocal(n^2\frac{\kappa_{Q}\omega^2}{\zeta})$. The total complexity of the IF-IPM using CGM is
$$\Ocal\left(\sqrt{n}\log\left(\frac{\mu^0}{\zeta}\right)\left(n^2\frac{\kappa_{Q}\omega^2}{\zeta}\right)\right).$$
The proof is complete.
\end{proof}
Similar to IF-QIPM using QLSA+QTA, we have a linear dependence on inverse precision, leading to an exponential complexity for finding an exact solution. We use the iterative refinement (IR) scheme in the next section to address this issue.
\section{Iterative Refinement for IF-IPM}\label{sec: IR}
To get an exact optimal solution, the time complexity contains an exponential term $\zeta=2^{-\Ocal(L)}$. To address this problem, we can fix $\zeta=10^{-2}$ and improve the precision by iterative refinement in $\Ocal(L)$ iterations~\cite{mohammadisiahroudi2022efficient}. Let us consider an LO problem in standard form with data $(A,b,c)$. Let $\nabla>1$ be a scaling factor. For a feasible solution $(x,y,s)\in \Pcal\Dcal$, we define the refining problem, as
\begin{equation} \label{eq: refining problem}
    \begin{aligned}
    \min_{\xhat}\  \nabla s^T&\xhat, \\
    {\rm s.t. }\;\;
    A\xhat &= 0, \\
    \xhat &\geq -\nabla x,
    \end{aligned}
    \qquad \qquad 
    \begin{aligned}
    \max_{\yhat,\shat} \  -\nabla x^T\shat,\ \ & \\
    {\rm s.t. }\;\;
    A^T\yhat +&\shat = \nabla s,\\
    &\shat \geq 0.
    \end{aligned}
\end{equation}
By changing variables, one can easily reformulate this problem to a standard LO.

 \begin{lemma}\label{lem: refining}
If $(\xhat,\yhat,\shat)$ is a $\hat{\zeta}$-optimal solution for refining problem~\eqref{eq: refining problem}, then $(x^r, y^r, s^r)$ is a $\frac{\hat{\zeta}}{\nabla^2}$-optimal solution for LO problem $(A,b,c)$ where $x^r=x+\frac{1}{\nabla} \hat{x}$, $y^r=y+\frac{1}{\nabla} \hat{y}$, and $s^r=c-A^Ty^r$
 \end{lemma}

 \begin{proof}
 It is straightforward to check that $(x^r, y^r, s^r)$ is a feasible solution for LO problem $(A,b,c)$. For the optimality gap, we have
 $$(x^r)^Ts^r=(x+\frac{1}{\nabla} \hat{x})^T(c-A^T(y+\frac{1}{\nabla} \hat{y}))=\frac{1}{\nabla^2}(\xhat+\nabla x)^T(\nabla s- A^T \yhat)\leq \frac{\hat{\zeta}}{\nabla^2}.$$
 The proof is complete.
 \end{proof}

Based on  Lemma~\ref{lem: refining}, we develop the IR-IF-IPM described in Algorithm~\ref{alg:iterative refinement}.

\begin{algorithm}[]
\caption{IR-IF-IPM / IR-IF-QIPM} \label{alg:iterative refinement}
\begin{algorithmic}[1]
\REQUIRE \big($A \in \mathbb{R}^{m\times n}, b\in \mathbb{R}^{m}, c\in \mathbb{R}^{n},\zeta<\hat{\zeta}<1$ \big)
\STATE  $k \gets 1$ and $\nabla_0\gets1$
\STATE $({x}^1, {y}^1, {s}^1) \gets$ \textbf{solve} $(A,b,c)$ using IF-IPM of Algorithm~\ref{alg: IF-IPM} or IF-QIPM of Algorithm~\ref{alg: IF-QIPM} with $\hat{\zeta}$ precision 
\WHILE{$(x^k,y^k,s^k)\notin \mathcal{PD}_{\zeta}$}
\STATE $\nabla^k \gets \frac{1}{\big(x^k\big)^Ts^k}$
\STATE $(\hat{x}^k, \hat{y}^k, \hat{s}^k) \gets$ \textbf{solve} $(A,0,\nabla^ks^{k})$ using IF-QIPM of Algorithm~\ref{alg: IF-IPM} or IF-QIPM of Algorithm~\ref{alg: IF-QIPM} with $\hat{\zeta}$ precision
\STATE $x^{k+1}\gets x^{k}+\frac{1}{\nabla^k} \hat{x}^k$ and $y^{k+1}\gets y^{k}+\frac{1}{\nabla^k}\hat{y}^k$
\ENDWHILE
\end{algorithmic}
\end{algorithm}
%

\begin{theorem}
The proposed IR-IF-IPM (IR-IF-QIPM) of Algorithm~\ref{alg:iterative refinement} produces a $\zeta$-optimal solution with at most $\Ocal(\frac{\log\log(\zeta)}{\log\log(\hat{\zeta})})$ inquiry to IF-IPM (IF-QIPM) with precision $\hat{\zeta}$. 
\end{theorem}

\begin{proof}
Based on Lemma~\ref{lem: refining}, we have
$$(x^{k+1})^Ts^{k+1}\leq \hat{\zeta}((x^{k})^Ts^{k})^2\leq \hat{\zeta}^{2^k-1},$$
because $(x^{1})^Ts^{1}\leq \hat{\zeta}$. Thus, we have $(x^{k+1})^Ts^{k+1}\leq \zeta$ for $k\geq \frac{\log\log(\zeta)}{\log\log(\hat{\zeta})}$.
\end{proof}

As Theorem 6.2 shows, the proposed IR-IF method enjoys quadratic convergence of the optimality gap.

\begin{corollary}
The total time complexity of finding an exact optimal solution for LO problems as in Definition~\ref{def: LO} with IR-IF-QIPM of Algorithm~\ref{alg:iterative refinement} with Algorithm~\ref{alg: IF-QIPM} as limited precision solver is $$\tilde{\Ocal}_{n,\kappa_Q,\omega,\mu^0,L}(n^{2.5}L\kappa_Q^2 \|Q\|\omega^5).$$ 
Furthermore, the total number of arithmetic operations for the IR-IF-IPM of Algorithm~\ref{alg:iterative refinement} with Algorithm~\ref{alg: IF-IPM} as limited precision solver is at most
$$\tilde{\Ocal}_{\mu^0,L}(n^{2.5}L\kappa_{Q}\omega^2).$$ 
\end{corollary}
The refining problem as described in \eqref{eq: refining problem} is not in the standard form. However by the change of variables 
$$(\xhat',\yhat',\shat'):=(\xhat+\nabla x,\yhat,\shat),$$
the standard reformulation of the refining problem is defined as  
\begin{equation} \label{eq: revised refining problem}
    \begin{aligned}
    \min_{\xhat}\  \nabla s^T&\xhat'- (\nabla)^2 s^Tx, \\
    {\rm s.t. }\;\;
    A\xhat' &= \nabla b, \\
    \xhat' &\geq 0,
    \end{aligned}
    \qquad \qquad 
    \begin{aligned}
    \max_{\yhat',\shat'} \  &\nabla b^T\yhat'- (\nabla)^2 s^Tx,\ \  \\
    {\rm s.t. }\;\;
    &A^T\yhat' +\shat' = \nabla s,\\
    &\shat' \geq 0.
    \end{aligned}
\end{equation}
At each iteration of IR, we are applying a feasible IPM to solve the refining problem \eqref{eq: revised refining problem}. Thus, we need an initial feasible interior solution for the refining problem. We are assuming that we have an interior feasible solution, $(x^0,y^0,s^0)$, for the original problem. The following theorem shows that the solution $(\nabla^k(x^0-x^k),\nabla^k(y^0-y^k),\nabla^ks^0)$ is a valid initial solution for IF-IPM to solve the refining problem at each iteration of IR methods.
\begin{theorem}\label{theo: intial for IR}
    Given an interior feasible solution $(x^0,y^0,s^0)$ for the original problem, and $(x^k, y^k, s^k)$ is the solution generated at iteration $k$ of the IR method, then $(\nabla^k(x^0-x^k),\nabla^k(y^0-y^k),\nabla^ks^0)$ is an interior feasible solution for the refining problem \eqref{eq: refining problem}.
\end{theorem}
\begin{proof}
    To check feasibility, we have
    \begin{align*}
        A (\nabla^k(x^0-x^k))&=\nabla^k( A x^0-A x^k)=\nabla^k(b-b)=0\\
        A^T (\nabla^k(y^0-y^k))+\nabla^k s^0&=\nabla^k( A^T (y^0-y^k)+s^0)=\nabla^k(c-A^Ty^k)=\nabla^ks^k.
    \end{align*}
    The solution $\nabla^k(x^0-x^k)$ is in the interior of the feasible region of the primal refining problem, since $x^0>0$ and $\nabla^k(x^0-x^k)>-\nabla^k x^k$. On the dual side, it is also strictly feasible since $\nabla^ks^0>0$. The proof is complete.
\end{proof}

Thus the initial solution for the standard reformulation \eqref{eq: revised refining problem} is $(\nabla^kx^0,\nabla^k(y^0-y^k),\nabla^ks^0)$. 
This solution has a similar distance to the central path as the initial solution for the original problem since the central path parameter $(\mu^0)^k=\frac{(\nabla^k)^2(x^0)^Ts^0}{n}=(\nabla^k)^2\mu^0$, and
$$\left\|\frac{(\nabla^k)^2(x^0)^Ts^0}{(\mu^0)^k}-e\right\|=\left\|\frac{(x^0)^Ts^0}{\mu^0}-e\right\|.$$
The coefficient matrix of the OSS system at the first iteration of IF-IPM (or IF-QIPM) at iteration $k$ of Algorithm~\ref{alg:iterative refinement} is
$$\begin{bmatrix}
    -\nabla^k X^0A^T& \nabla^k S^0V
\end{bmatrix}=\nabla^k \begin{bmatrix}
    -X^0 A^T& S^0 V
\end{bmatrix}.$$
Thus, the condition number of the first OSS system is the same for all iterations of the IR method. 
In the next section, we discuss how we apply IR-IF-IPM to the SDE formulation when we do not have an interior feasible solution for the original problem.

\section{IF-IPM for SDE Model} \label{sec: self-dual}
%
The proposed IF-IPM requires an initial feasible interior solution.
In practice, such a feasible interior solution is not available.
In this case, the SDE formulation \cite{Roos2005_Interior, Ye1994_iteration} can be used where an all-one vector $e$ is a feasible interior point. 
The canonical formulation is more appropriate for the direct use of the SDE formulation in the proposed IF-IPM. 
We use the conical formulation of the LO problem as in Definition~\ref{def: LO: conocical}. 
We can derive the self-dual model as
\begin{equation}\label{eq: self-dual}
\begin{aligned}
    &\min \quad (n'+m'+2)\gamma \\
     &\begin{matrix}
    \st & &A'x &-b'\tau  &+\bar{b}\gamma&\geq 0,\phantom{(n'+m'+1)}\\
    &-{A'}^Ty & &+c'\tau&+ \bar{c}\gamma&\geq 0,\phantom{(n'+m'+1)}\\
    &{b'}^T y& -{c'}^T x & & + \bar{o}\gamma&\geq 0,\phantom{(n'+m'+1)}\\
    &-\bar{b}^T y&-\bar{c}^T x&-\bar{o} \tau& &\geq  -(n'+m'+2),\\
    \end{matrix}\\
    &\qquad x\geq 0,\tau\geq 0,y\geq 0{,\; \rm  and }\; \gamma\geq 0,
\end{aligned}
\end{equation}
where $\bar{b}=b'-A'e+e$, $ \bar{c}={A'}^Te+e-c'$, and $\bar{o}=1+{c'}^Te-{b'}^Te$. 
One can verify that the dual problem of~\eqref{eq: self-dual} is itself~\cite{Roos2005_Interior}. 
We write problem~\eqref{eq: self-dual} in the standard format by introducing slack variables $(u,s,\phi,\rho)$ as
\begin{equation}\label{eq: self-dual standard}
\begin{aligned}
    &\min\quad (n'+m'+2)\gamma \\
    &\begin{matrix}
    \st & &A'x &-b'\tau  &+\bar{b}\gamma&-u&= 0,\phantom{(n'+m'+1)}\\
    &-{A'}^Ty & &+c'\tau&+ \bar{c}\gamma&-s&= 0,\phantom{(n'+m'+1)}\\
    &{b'}^T y& -{c'}^T x & & + \bar{o}\gamma&-\phi&= 0,\phantom{(n'+m'+1)}\\
    &-\bar{b}^T y&-\bar{c}^T x&-\bar{o} \tau& &-\rho&=  -(n'+m'+2),\\
    \end{matrix}\\
    &\qquad x\geq 0,\tau\geq 0,y\geq 0, \gamma \geq 0,s\geq 0,u\geq 0, \phi \geq 0, \rho \geq 0.
\end{aligned}
\end{equation}
One can verify that $(y^0,x^0,\tau^0,\gamma^0,u^0,s^0,\phi^0,\rho^0)=(e,e,1,1,e,e,1,1)$ is a feasible interior solution of problem~\eqref{eq: self-dual standard}.
Based on the Strong Duality Theorem \cite{dantzig2016linear,Roos2005_Interior}, any optimal solution satisfies $\gamma=0$ and
$$x^Ts+y^T u+\tau \phi +\rho \gamma=0.$$ 
%
 
%
\begin{theorem}[\cite{Roos2005_Interior}] \label{theorem: optimal of self-dual}
 The following statements hold for model~\eqref{eq: self-dual}.
 \begin{enumerate}
     \item Problem~\eqref{eq: self-dual} has a strictly complementary optimal solution $(y^*,x^*,\tau^*, \gamma^*,\allowbreak u^*, s^*,\phi^*,\rho^*)$ such that $\gamma^*=0$, $x^*+s^*>0$, $y^*+u^*>0$, $\tau^*+\phi^*>0$, and $\rho^*>0$.
     \item If $\tau^*>0$, $(\frac{x^*}{\tau^*},\frac{y^*}{\tau^*},\frac{s^*}{\tau^*},\frac{u^*}{\tau^*})$ is a strictly complementary optimal solution of the original LO problem.
     \item If $\tau^*=0$, ${c'}^Tx^*<0$, and ${b'}^Ty^*\leq0$, then the original dual problem is infeasible, and if the original primal problem is feasible, then it is unbounded.
     \item If $\tau^*=0$, ${c'}^Tx^*\geq0$, and ${b'}^Ty^*>0$, then the original primal problem is infeasible, and if the original dual problem is feasible, then it is unbounded.
     \item If $\tau^*=0$, ${c'}^Tx^*<0$, and ${b'}^Ty^*>0$, then both original primal and dual problems are infeasible.
 \end{enumerate}
\end{theorem}
The feasible Newton system for this formulation is
\footnotesize
\begin{equation}\label{eq: newton system of self-dual}
\arraycolsep=1.3pt
\begin{aligned} 
    \begin{array}{rrrrrrrrl}
     &A'\Delta x^k &-b'\Delta \tau^k  &+\bar{b}\Delta\gamma^k&-\Delta u^k&&&&= 0,\\
    -{A'}^T\Delta y^k & &+c'\Delta \tau^k &+ \bar{c}\Delta\gamma^k&&-\Delta s^k&&&= 0,\\
    {b'}^T \Delta y^k& -{c'}^T \Delta x^k & & + \bar{o}\Delta \gamma^k&&&-\Delta \phi^k&&= 0,\\
    -\bar{b}^T \Delta y^k&-\bar{c}^T \Delta x^k&-\bar{o} \Delta \tau^k  &&&& &-\Delta\rho^k&= 0,\\
     & S^k \Delta x^k &&&&+ X^k\Delta s^k&&& = \beta \mu^k e- X^k s^k,\\
     U^k\Delta y^k&&&&+Y^k\Delta u^k&&&&=\beta \mu^k e -Y^k u^k, \\
     &&\phi^k\Delta\tau^k&&&& +\tau^k\Delta\phi&& = \beta\mu^k- \tau^k\phi^k, \\
     &&&\rho^k\Delta\gamma^k &&&& +\gamma^k\Delta\rho^k&=\beta\mu^k-\gamma^k\rho^k.
    \end{array}
\end{aligned}
\end{equation}
\normalsize
To derive the OSS system for Newton system~\eqref{eq: newton system of self-dual}, we define
\begin{equation}\label{eq: self-dual matrices}
\begin{aligned}
    \mathcal{A}&=\begin{bmatrix}
    I &\mathbf{0}&\mathbf{0} &\mathbf{0}&\mathbf{0} &-A'&b' &-\bar{b}\\
    \mathbf{0}&I&\mathbf{0}&\mathbf{0}&{A'}^T&\mathbf{0}&-c&-\bar{c}\\
    \mathbf{0}&\mathbf{0}& 1 &\mathbf{0}&-{b'}^T &{c'}^T &\mathbf{0}&-\bar{o}\\
    \mathbf{0}&\mathbf{0}&\mathbf{0}& 1 &\bar{b}^T&\bar{c}^T &\bar{o}^T &\mathbf{0} \\
    \end{bmatrix},\ \mathcal{R}=\begin{bmatrix}
    \beta \mu^k e-Y^ku^k\\
    \beta \mu^k e- X^ks^k \\
    \beta \mu^k-\tau^k\phi^k\\
    \beta \mu^k-\gamma^k\rho^k
    \end{bmatrix},\\
    \mathcal{D}&=\begin{bmatrix}
    Y^k&\mathbf{0}&\mathbf{0}&\mathbf{0}&U^k&\mathbf{0}&\mathbf{0}&\mathbf{0}\\
    \mathbf{0}&X^k&\mathbf{0}&\mathbf{0}&\mathbf{0}&S^k&\mathbf{0}&\mathbf{0}\\
    \mathbf{0}&\mathbf{0}&\tau^k&\mathbf{0}&\mathbf{0}&\mathbf{0}&\phi^k&\mathbf{0}\\
    \mathbf{0}&\mathbf{0}&\mathbf{0}&\gamma^k&\mathbf{0}&\mathbf{0}&\mathbf{0}&\rho^k
    \end{bmatrix},\\
    \Delta \mathcal{X}&=(\Delta u^k; \Delta s^k; \Delta \phi ^k; \Delta\rho^k; \Delta y^k; \Delta x^k; \Delta \tau^k; \Delta \gamma^k),
\end{aligned}
\end{equation}
where $\mathbf{0}$ is the all-zero matrix, and ``;" indicates that the corresponding column vectors are vertically concatenated. 
Then, the Newton system~\eqref{eq: newton system of self-dual} can be simplified as
\begin{equation}\label{eq: newton system of self-dual2}
\begin{aligned}
   \Delta \mathcal{X} \in \text{Null}(\mathcal{A}),\ \mathcal{D}\Delta \mathcal{X}=\mathcal{R}.
\end{aligned}
\end{equation}
The basis for the null space of $\mathcal{A}$, as $\mathcal{A}$ includes an identity matrix, is directly given by the columns of 
$$
    \mathcal{V}=
    \begin{bmatrix}
    \mathbf{0}  &-A'        &b'         &-\bar{b}\\
    {A'}^T      &\mathbf{0} &-c'        &-\bar{c}\\
    -{b'}^T     &{c'}^T     &\mathbf{0} &-\bar{o}\\
    \bar{b}^T   &\bar{c}^T  &\bar{o}^T  &\mathbf{0} \\
    -I          &\mathbf{0} &\mathbf{0} &\mathbf{0}\\
    \mathbf{0}  &-I         &\mathbf{0} &\mathbf{0}\\
    \mathbf{0}  &\mathbf{0} & -1        &\mathbf{0}\\
    \mathbf{0}  &\mathbf{0} &\mathbf{0} & -1 
    \end{bmatrix}.
$$
In this case, our method does not require Gaussian elimination to derive a basis for the null space of the coefficient matrix, which saves preprocessing cost. 
The OSS for this formulation is
\begin{equation}\label{eq: OSS of self-dual}
\mathcal{D}\mathcal{V}\mathcal{\lambda}=\mathcal{R},
\end{equation}
where $\mathcal{\lambda}\in\mathbb{R}^{n'+m'+2}$ and the size of the system will be $n'+m'+2$. 
We can calculate the Newton direction by $\Delta \mathcal{X}=\mathcal{V}\mathcal{\lambda}$. 
Considering $\tilde{\lambda}$ as an inexact solution of the system~\eqref{eq: OSS of self-dual}, the inexact solution $\widetilde{\Delta \mathcal{X}}$ is a feasible direction since $\widetilde{\Delta \mathcal{X}}\in \text{Null}(\mathcal{A})$. 
A convergent IF-IPM requires $\|r\|\leq \eta \mu$ where $r=\mathcal{D}\mathcal{V}\tilde{\lambda}-\mathcal{D}\mathcal{V}\lambda$. 
Thus, the error bound $\epsilon=\frac{\eta\mu}{\|\mathcal{D}\mathcal{V}\|}$ is needed.
%
\begin{lemma}\label{lemma: oss for self-dual}
Let $( u,  s, \phi ,\rho,  y, x, \tau, \gamma)\in \mathcal{PD}^0$, then the following statements hold.
\begin{enumerate}
    \item Systems~\eqref{eq: OSS of self-dual} and~\eqref{eq: newton system of self-dual} are equivalent.
    \item System~\eqref{eq: OSS of self-dual} has a unique solution.
    \item Any solution of system~\eqref{eq: OSS of self-dual} satisfies $$(\Delta x)^T\Delta s+(\Delta y)^T\Delta u+\Delta \tau \Delta \phi+\Delta \gamma \Delta \rho =0.$$
\end{enumerate}
\end{lemma}
%

The IF-IPM of Algorithm~\ref{alg: IF-IPM} (or IF-QIPM of Algorithm~\ref{alg: IF-QIPM}) can now be applied to the SDE formulation.
%
\begin{theorem}\label{theorem: convergence of IF-IPM for self-dual}
 For IF-IPM of Algorithm~\ref{alg: IF-IPM} (or IF-QIPM of Algorithm~\ref{alg: IF-QIPM}) applied to the SDE formulation, the following statements hold.
 \begin{enumerate}
     \item The sequence $\{\mu_k\}_{k\in \mathbb{N}}$ converges linearly to zero.
     \item For any $k\in \mathbb{N}$, $\mathcal{X}^k\in \mathcal{N}(\theta)$.
     \item After $\Ocal(\sqrt{n'}\log(\frac{1}{\zeta}))$ iterations, $\mathcal{X}^k\in \mathcal{PD}_{\zeta}$.
     \item The iteration complexity of finding an exact optimal solution is $\Ocal(\sqrt{n'}L)$.
 \end{enumerate}
\end{theorem}
%
 A similar analysis as Section~\ref{sec: IF-QIPM} can be conducted here.
 Thus, the total time complexity of the IR-IF-QIPM applied to the SDE formulation is 
$$\tilde{\Ocal}_{n',\kappa_\mathcal{V},\omega,\mu^0,L}({n'}^{2.5}L\kappa_{\mathcal{V}}^2 \|\mathcal{V}\|\omega^5).$$ 
Furthermore, total arithmetic operations for IF-IPM with CGM is
$$\tilde{\Ocal}_{\mu^0,L}({n'}^{2.5}L\kappa_{\mathcal{V}}\omega^2).$$ 
 It is clear that matrix $\mathcal{V}$ is larger and denser than $Q$; however, $\mathcal{V}$ is only a constant factor larger and denser than $Q$.
 Therefore, the time complexity of solving the SDE formulation in $\tilde{\Ocal}$ notation is the same as that of the original problem with an initial interior feasible solution.

\section{Numerical Experiments}\label{sec: numerical}

The proposed IF-IPMs and IF-QIPMs are implemented in the Python programming language\footnote{\url{https://github.com/qcol-lu/qipm}}. 
Our implementation can be used with both classical and quantum linear system solvers. 
In this section, first, we investigate the performance of quantum devices and simulators to solve linear systems of equations.
To this aim, we report some experiments with IBM and Quantinuum quantum devices and simulators as they use different technologies.
In the next part, we present some illustrative numerical experiments for the proposed IR-IF-QIPM using a quantum simulator of QLSA, a quantum error simulator of QLSAs, and a classical CGM.

\subsection{QLSAs on different quantum devices and simulators}

QLSAs are among the quantum algorithms that require fault-tolerant quantum computers, see also Quantum Fourier Transform, and Shor's Algorithm. 
It is expected that current noisy intermediate-scale quantum (NISQ) devices are not capable of executing QLSAs circuits, as NISQ devices can handle circuits with only limited circuit depth and width. 
However, in this part, we want to explore these barriers and quantify what type of linear systems can be solved on contemporary quantum devices and simulators.

The QLSA used in the analysis of Section~\ref{sec: IF-QIPM} uses QRAM and block-encoding of matrices. 
However, to date, there is no physical implementation of QRAM, and implementing block-encoding efficiently is not possible on current NISQ devices.
Thus we use simple versions of HHL algorithm implemented using QISKIT.
Here, we report numerical experiments of different implementations using IBM and Quantinuum real machines and simulators.
The reason is that these two companies use different technologies.
IBM’s quantum computers are based on superconducting qubits, but Quantinuum uses trapped-ion qubits. 
Although superconducting qubits can perform computations faster, trapped-ion qubits have much longer coherence times and all-to-all connectivity.

\subsubsection{QLSA on IBM Machines and Simulators}
Table~\ref{tab:HHL-real-exp} demonstrates the performance of solving three random linear systems using a simple implementation of HHL algorithm\footnote{\url{https://github.com/Firepanda415/quantum_linear_solvers}}. Using IBMQ Machines and Simulator, we can solve linear systems with dimension $n\leq 8$ and $\kappa\leq 10$. By increasing dimension and condition number the circuit depth increases and cannot be executed on current quantum hardware. Additionally, the solutions provided to small linear systems have considerable error, thus impractical for QIPMs.
\begin{table}[]
\centering
\scriptsize
\resizebox{\textwidth}{!}{\begin{tabular}{|c|c|cc|cc|cc|c|} \hline 
 \multirow{2}{*}{Dimension} & \multirow{2}{*}{Sparsity} & \multicolumn{2}{c|}{Number of gates} & \multicolumn{2}{c|}{Circuit Depth} & \multicolumn{2}{c|}{Relative Error}  & \multirow{2}{*}{Machine} \\                        &                           & Simu.            & Device          & Simu.            & Device    & Simu.            & Device         &                                             \\  \hline 
2                   & 0                         & 90                   & 115               & 397          & 576        & 0.0394             & 0.239           & osaka                \\
4                   & 0.5                       & 425                  & 593               & 2089         & 2817       & 0.0155             & 0.465            & osaka                \\
8                  & 0.78                      & 2318                 & 4339              & 11257        & 19688      & 0.0464             & 0.401     & brisbane           \\ \hline
\end{tabular}}
\caption{Solving three random systems with HHL implementation on IBMQ devices and simulator.}\label{tab:HHL-real-exp}
\end{table}
\subsubsection{QLSA on Quantinuum Machines and Simulators}
We also examine Quantinuum machines and simulators, as trapped-ion qubits have a longer coherence time and the potential to execute longer circuits. Table~\ref{tab quantinuum} presents the results of solving three random linear systems using an implementation of the HHL algorithm\footnote{\url{https://github.com/QCOL-LU/QLSAs}}. As results show the Quantinuum machines have similar performance for solving linear systems of equations. This result further confirms that contemporary NISQ devices are capable of executing QLSAs only for very small linear systems.
\begin{table}[]
\centering
\scriptsize
\resizebox{\textwidth}{!}{\begin{tabular}{|c|c|c|c|c|c|c|}
\hline
 Dimension & Condition Number  & \# Qubits & Circuit Depth & Total Gates & Relative Error \\ \hline
 2 & 1.42 &  4 & 25 & 36  & 0.37 \\ 
 4 & 5.73 &  7 & 359 & 547  & 0.43 \\ 
 8 & 3.96 &  12 & 3538 & 5123 & 0.32 \\ \hline
\end{tabular}}
\caption{Solving three random systems with HHL implementation on Quantinuum H1-1E}
\label{tab quantinuum}
\end{table}
\subsubsection{QISKIT HHL Simulator}
This section provides numerical results for  the QISKIT AQUA quantum simulator. 
The numerical results are run on a workstation with Dual Intel Xeon{\textregistered} CPU E5-2630~@ 2.20 GHz (20 cores) and 64 GB of RAM. 
For the computational experiments, we have developed a Python \texttt{QIPM} package available for public use\footnote{\url{https://github.com/qcol-lu/qipm}}.

IBM has implemented a QLSA, which is similar to the HHL method, without block-encoding and QRAM.
As discussed in previous sections, with the current technology, the number of available qubits in gate-based quantum computers is limited.
One of the main issues with contemporary, NISQ era, quantum computers is that they are not scalable compared to classical computers. 
Currently, larger NISQ devices suffer more from increasing noise and error. 
On the other hand, quantum simulator algorithms are computationally expensive, thus the maximum number of Qubits in a quantum simulator is limited to less than 50.
The main advantage of using a quantum simulator is that we do not need to handle the noise of NISQ devices. 
Despite this, we still need to deal with a high error level due to insufficient qubits and the high cost of QLSA+QTA in order to find high-quality solutions.

We used two post-processing procedures to improve QLSA's performance.
\begin{enumerate}
    \item We first scale the linear system solution such that it satisfies the  equality~$ \|Mz\|=\|\sigma\|$. 
    \item We also check the sign of the linear system solution by comparing it with its negative.
\end{enumerate}
The condition number increases the solution time.
Notably, the dimension of the linear system increases the solution time as a step function for building the quantum circuit.
The dimension of the system must be a power of two.
The simulator expands the system size to the smallest possible power of two.
Table~\ref{tab: condition number and size of quantum circuit} presents the number of qubits needed in quantum circuits for achieving the same precision for a $4x4$ linear systems of equations.
The QISKIT simulator error oscillates between zero and the norm of an actual solution.  
In some cases, the IBM QISKIT simulator fails, and it reports a zero vector as a solution. 
While the simulator has a parameter for tuning precision, it frequently does not reach the predefined precision.
We could not find any meaningful relationship between the precision of the simulator, the condition number of the coefficient matrix, and the dimension of the system.

\begin{table}[]
\caption{Size of the circuit for linear systems with different condition numbers ($N=4$ and $\epsilon\leq10^{-2}$) }
\label{tab: condition number and size of quantum circuit}
\centering
\begin{tabular}{r|c c c c c c c c c c c }
\hline
$\kappa$ &  $2^1$ & $2^2$ & $2^3$ & $2^4$ &$2^5$ &$2^6$ &$2^7$ & $2^8$& $2^9$ & $2^{10}$ & $2^{11}$\\ \hline
   \#qubits &   6  & 6 & 7 & 8 & 9  & 10 & 11 & 12 & 12 & 13 & 15\\\hline
\end{tabular}
\end{table}
This simulator has a somewhat different configuration from the theoretical HHL algorithm. There is no direct parameter for tuning precision. One of the parameters that highly affects the precision is the time of the Hamiltonian simulation $t$. The best value is $t= \frac{\pi}{\lambda_{\max}}$ where the simulator can almost find solutions with a small error for the scaled system. Another value for this parameter leads to considerable error, affecting the convergence of QIPMs. Finding the maximum eigenvalue is not computationally expensive by using the Power Iteration method. In our implementation, we used $t= \frac{\pi}{\lambda_{\max}}$ where the maximum eigenvalue of the NES $\lambda_{\max}$ is calculated by the Power Iteration method. We also need to tune other parameters to increase the precision of HHL simulator and decrease the time of solution. Another parameter affecting the performance of the HHL simulator is the number of ancillae qubits. Figure~\ref{fig: ancila} depicts the time and error of solving randomly generated systems with different numbers of ancillae qubits. We can see that by increasing the number of ancillae qubits, the error decreases, but the time increases. For this case, the appropriate number of ancillae qubits is four, since after that, the error decreases slightly, but time increases rapidly.
\begin{figure}[]
    \centering
    \includegraphics[scale=0.35]{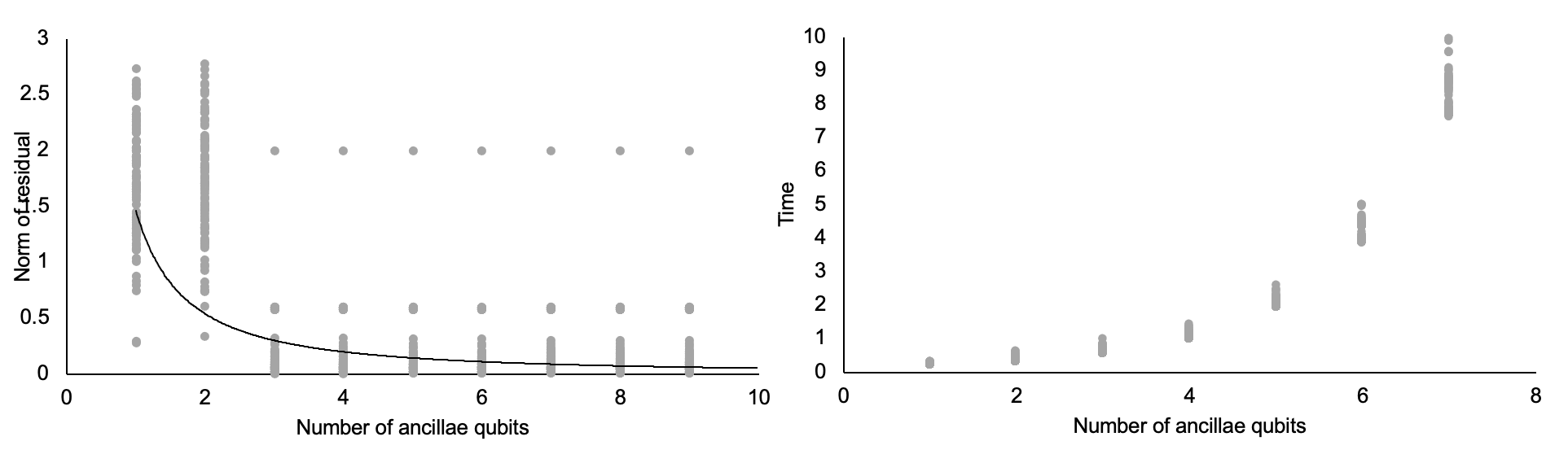}
    \caption{Effect of the number of ancillae qubits on the error and time of the HHL simulator}
    \label{fig: ancila}
\end{figure}
It is necessary to evaluate the performance of HHL simulator based on features of the linear system, such as the condition number, the norm of the coefficient matrix, and the RHS. As expected, Figure~\ref{fig: dimension} shows that the dimension of the system does not affect the error, but the time of solution is increasing like a step function since the HHL simulator builds a system where the dimension is a power of two.
\begin{figure}[]
    \centering
    \includegraphics[scale=0.4]{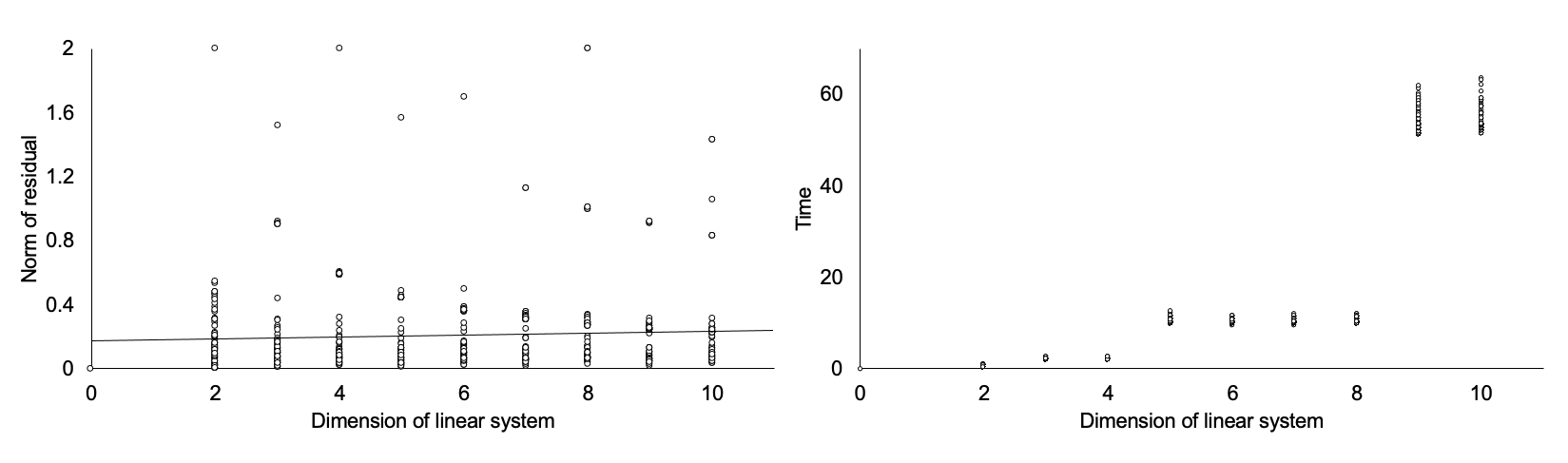}
    \caption{Effect of dimension on the time and the error of the HHL simulator}
    \label{fig: dimension}
\end{figure}
To confirm the theoretical result, Fig~\ref{fig: RHS} illustrates that the norm of the RHS affects the error due to scaling, but it does not affect the time of solution. In theoretical results, it affects the time complexity since we fixed the target error, and QLSAs need time proportional to the norm of the RHS. Contrary to the theoretical analysis of the HHL method, the HHL simulator is almost independent on the condition number of the coefficient matrix as depicted in Figure~\ref{fig: condition number}. The reason is that the number of iterations of the HHL method is $\Ocal(\kappa)$ to get a solution with high probability, and in this simulator, we do not need to replicate the computation since the probability can be calculated numerically.  Figure~\ref{fig: norm of matrix} indicates that the norm of the coefficient matrix does not affect the time and the error of the solution, and indicates that this simulator does not properly represent the theoretical complexity of QLSAs.
\begin{figure}[]
    \centering
    \includegraphics[scale=0.4]{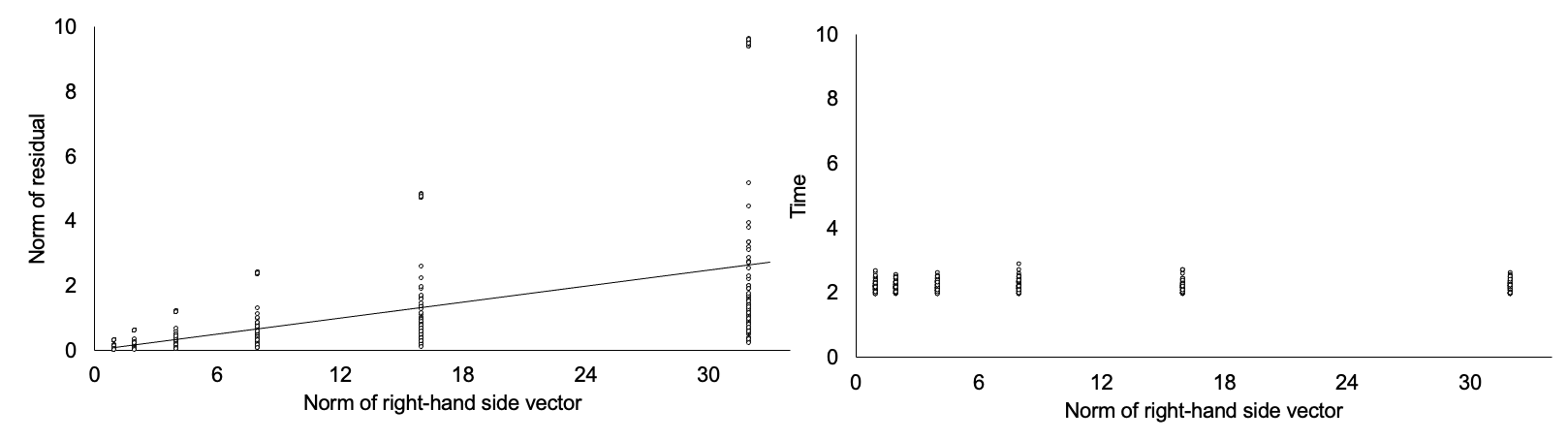}
    \caption{Effect of the norm of RHS on the time and the error of the HHL simulator}
    \label{fig: RHS}
\end{figure}
\begin{figure}[]
    \centering
    \includegraphics[scale=0.4]{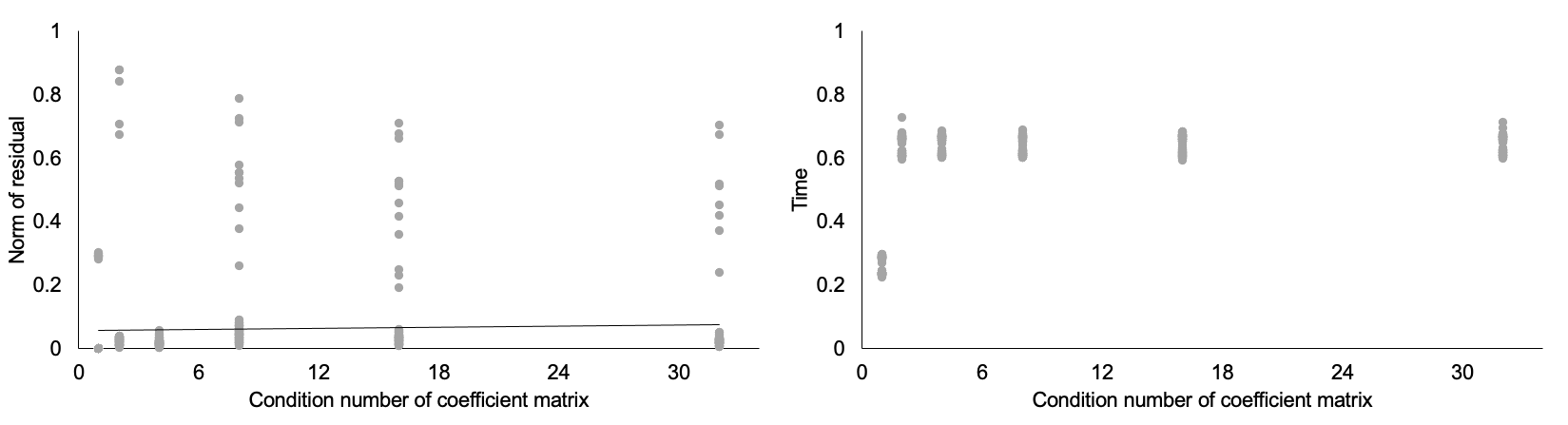}
    \caption{Effect of the condition number on the time and the error of the HHL simulator}
    \label{fig: condition number}
\end{figure}
\begin{figure}[]
    \centering
    \includegraphics[scale=0.38]{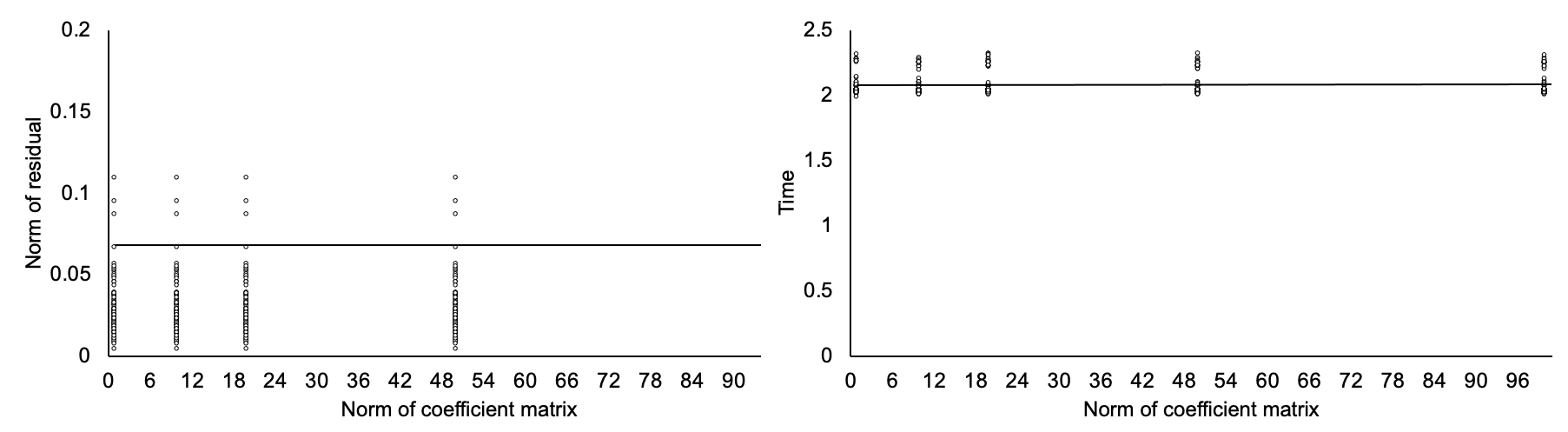}
    \caption{Effect of the norm of coefficient matrices on the time and the error of the HHL simulator}
    \label{fig: norm of matrix}
\end{figure}

\subsection{IF-QIPM using QISKIT Simulator}

As discussed earlier, current quantum simulators can only solve linear systems with limited dimensions and condition numbers. Due to several factors, implementing the proposed IF-QIPM on existing quantum devices and simulators is challenging. Since the OSS system is asymmetric, constructing a larger system with a Hermitian coefficient matrix is necessary to apply QLSAs. Additionally, the absence of an initial interior solution necessitates the use of SDE formulations, which increase the dimensionality of Newton systems too. For instance, solving a linear optimization problem with just two variables would require handling a linear system of dimension 8 when one use QLSAs. Moreover, the errors inherent in simulators and the high condition numbers of Newton systems pose significant challenges for IF-QIPMs on current platforms.

In Figure~\ref{fig:log}, we present the results of solving a randomly generated linear optimization (LO) problem using the Qiskit HHL simulator. The problem is designed with a predetermined interior solution, eliminating the need for a SDE. As shown, the condition number increases throughout the iterations of IF-QIPM, leading to a significant rise in the simulation time of QLSAs. Despite these challenges, IF-QIPM effectively handles errors, maintaining primal and dual infeasibility close to zero. Moreover, within just a few iterations, the solution achieves a precision of $10^{-2}$. 
\begin{figure}[]
    \centering
    \includegraphics[width=0.8\linewidth]{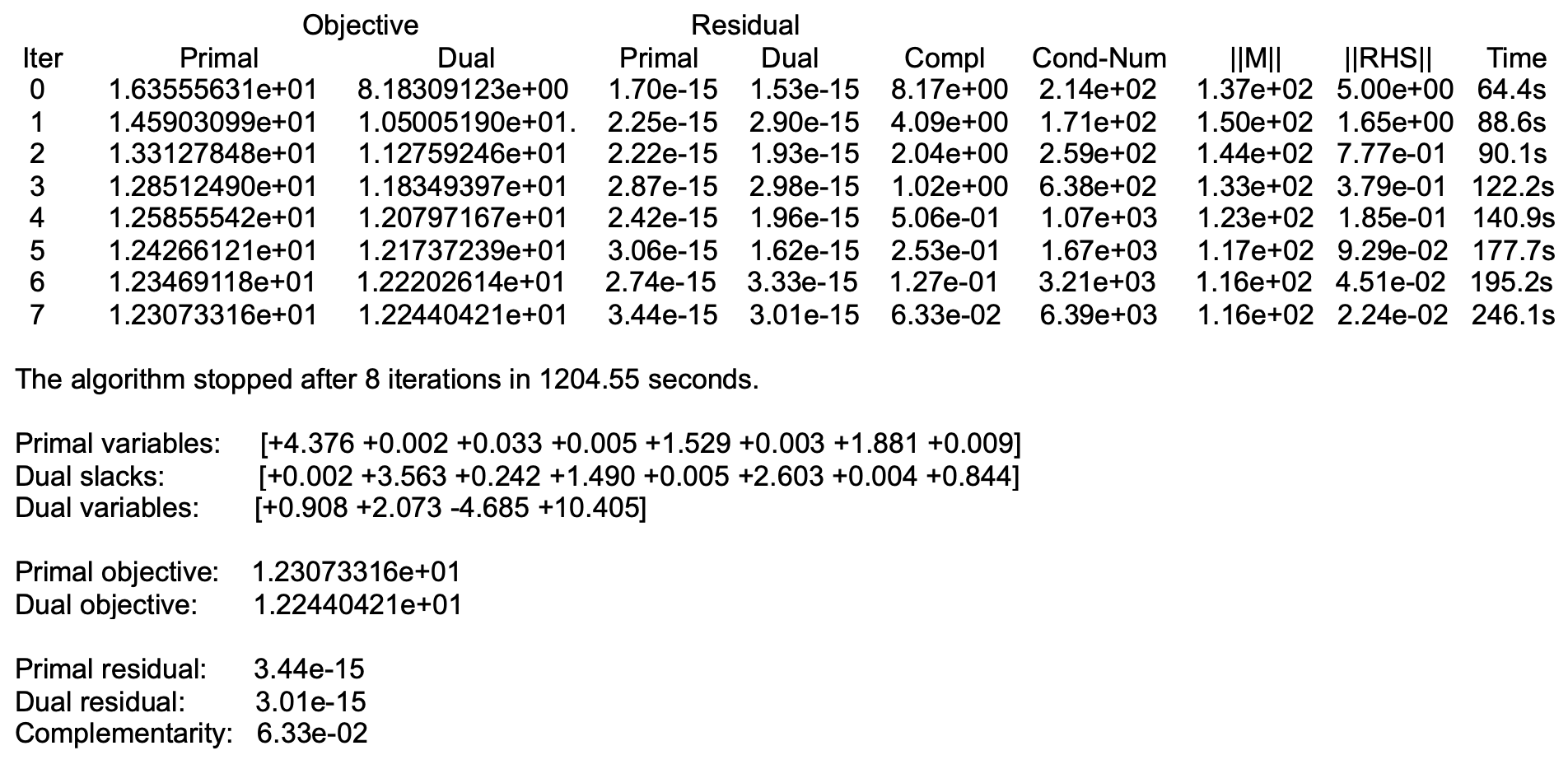}
    \caption{Result of solving a LO problem with 8 variables using QISKIT Simulator}
    \label{fig:log}
\end{figure}

\subsection{Analysis of IR-IF-IPM}
The size of the linear systems and their asymmetric nature in the IF-IPM make it hard to study the performance of IF-QIPM using this quantum solver in our experiments.
In this section, we examine IR-IF-IPM with a classical inexact linear solver.
We conducted the numerical results on a workstation with Dual Intel Xeon{\textregistered} CPU E5-2630~@ 2.20 GHz (20 cores) and 64 GB of RAM. 
We employ the random instance generator of~\cite{generator} to generate LO instances systematically and with desired parameters.
We generated feasible LO instances in the canonical format with $m'=4$, $n'=12$, $\kappa_{A'}=4$ and $\|A'\| = \|b'\| = \|c'\| = 2$.
We also set $\eta = 0.1$ and $\zeta= 10^{-6}$.
Here, we assume that $\zeta$ is the desired precision of the obtained optimal solution of the canonical problem.

\begin{figure}[] \centering
\includegraphics[width =0.5 \linewidth, trim = {.1cm .1cm .1cm .1cm}, clip]{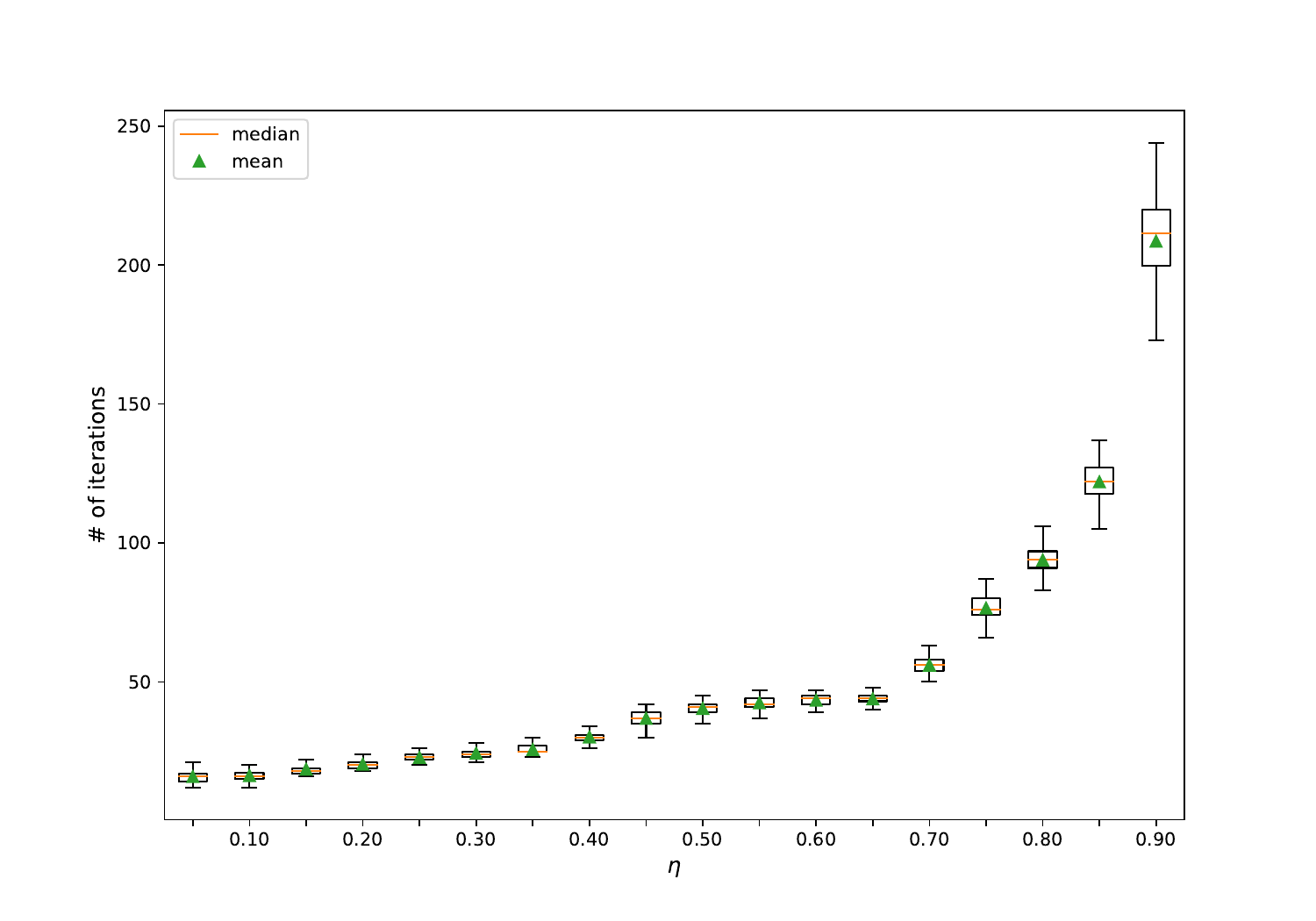}
\caption{The effect of the solver error on the number of iterations of the SDE IF-IPM.} \label{fig: eta_iter}
\end{figure}

Fig.~\ref{fig: eta_iter} illustrates how the noise of the linear system solver affects the number of iterations of the proposed SDE IF-IPM.
The number of iterations of the algorithm increases as the error level increases.
However, the algorithm is relatively stable with respect to $\eta$ while it is less than $0.7$.
The number of iterations increases rapidly as the error parameter converges to $1$.

As the IF-IPM converges to the central optimal solution, the condition number of linear systems solved at every iteration of the algorithm grows to infinity.
The largest condition number of the OSS occurs at the last iterations of IF-IPM.
Fig.~\ref{fig: largest condition number of the self-dual embedding IF-IPM} shows how the largest condition number of the OSS changes in the self-dual IF-IPM with respect to different parameters of the problem and the algorithm.
Fig.~\ref{fig: min_sing_A_cond} indicates that the condition number of the OSS is almost indifferent to the smallest singular value of $A'$.
Here, we adjust the smallest singular value of the matrix $A'$ by just changing its condition number while keeping its norm constant. 
In other words, the condition number of the OSS does not depend directly on the condition number of matrix $A'$, which is one of its submatrices.
Conversely, the largest singular values of OSS submatrices, i.e., the $\ell_2$ norm of those submatrices, affect the OSS condition number.
Fig.~\ref{fig: norm_A_cond} and Fig.~\ref{fig: norm_b_cond} show that the norm of matrix $A'$ and right-hand side (RHS) vector $b'$ have a direct relationship with the condition number of the OSS.
A similar trend can be observed with vector $c'$.
As shown in Fig.~\ref{fig: pre_cond}, the condition number of the OSS changes with rate $\frac{1}{\zeta}$. 
This observation can be explained by the relationship derived in~\eqref{eq: cond_oss dependence}.

\begin{figure}[] \centering
\subfloat[The smallest singular value of matrix $A'$]{\label{fig: min_sing_A_cond}\includegraphics[width =0.5 \linewidth, trim = {.1cm .1cm .1cm .1cm}, clip]{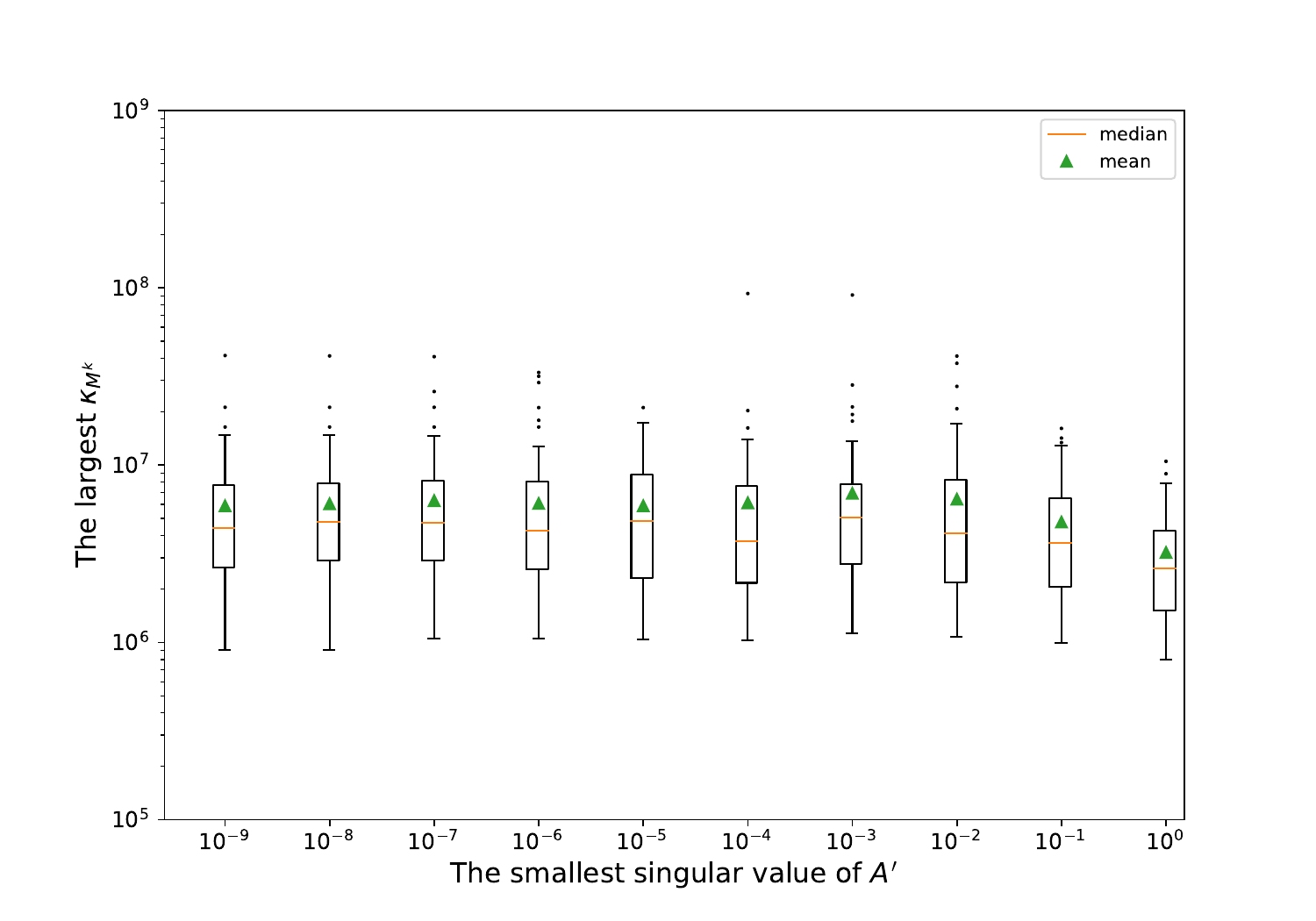}}
\subfloat[The norm of matrix $A'$]{\label{fig: norm_A_cond}\includegraphics[width =0.5 \linewidth, trim = {.1cm .1cm .1cm .1cm}, clip]{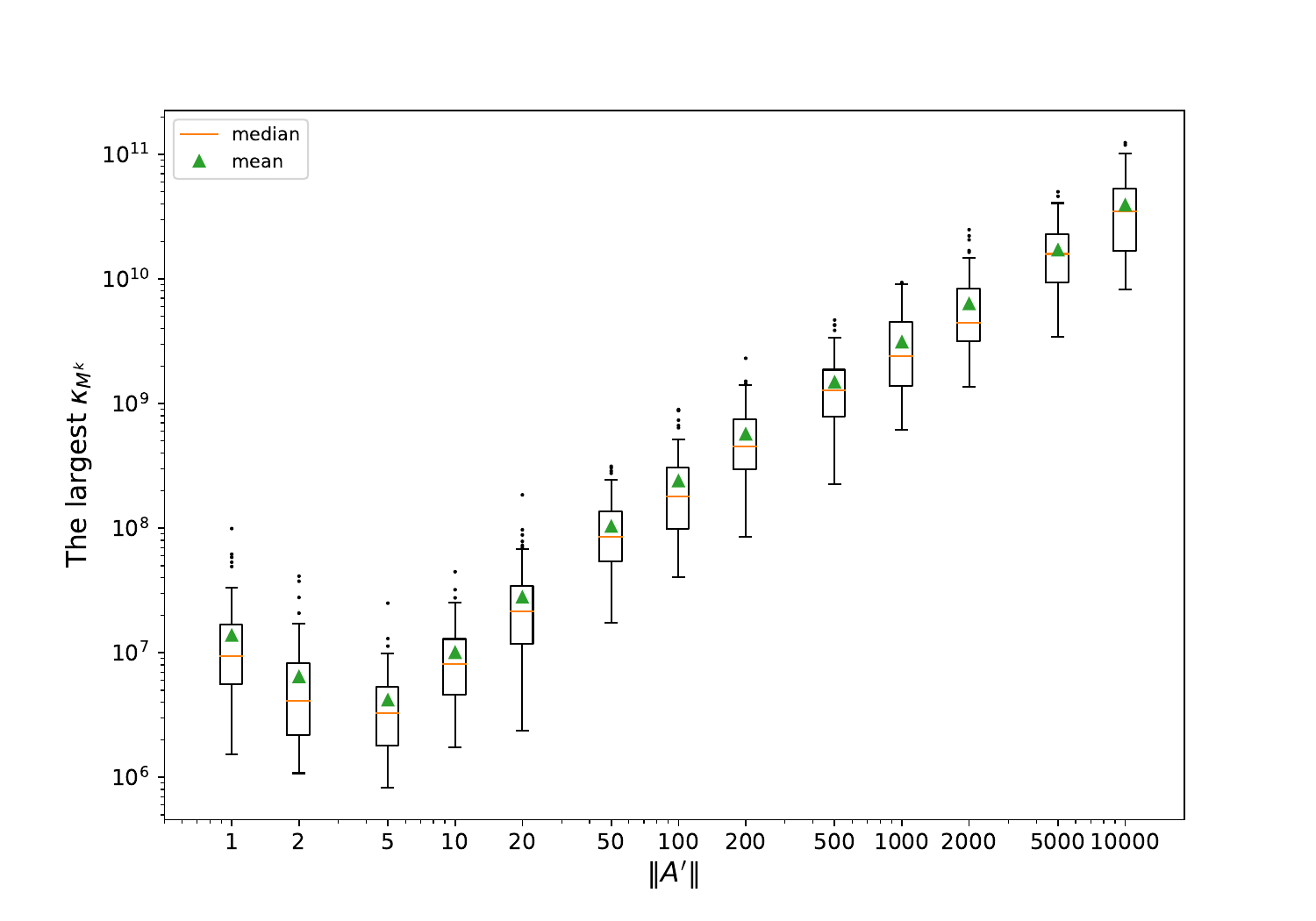}}\\
\subfloat[The desired precision]{\label{fig: pre_cond}\includegraphics[width =0.5 \linewidth, trim = {.1cm .1cm .1cm .1cm}, clip]{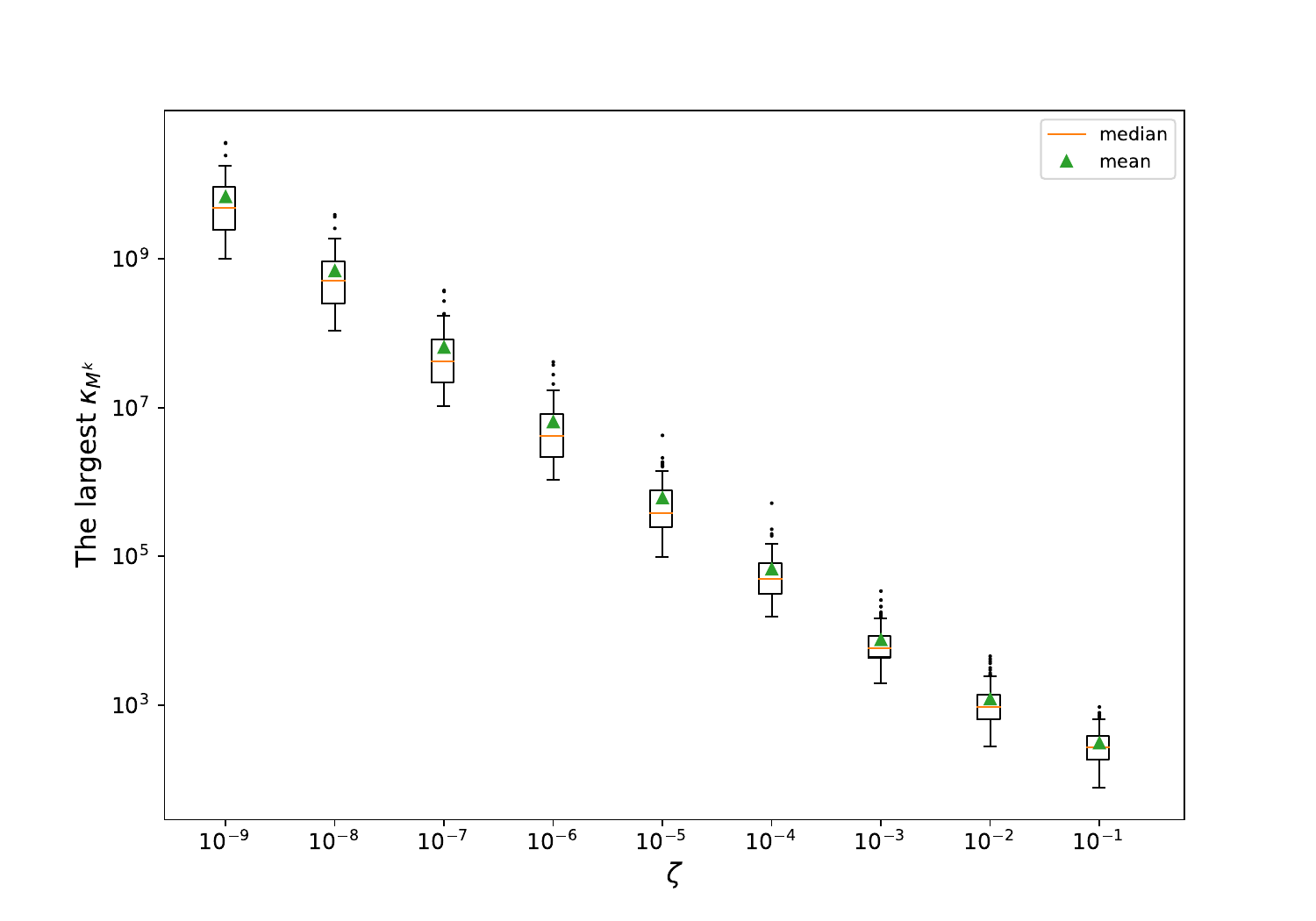}}
\subfloat[The norm of RHS vector $b'$]{\label{fig: norm_b_cond}\includegraphics[width =0.5 \linewidth, trim = {.1cm .1cm .1cm .1cm}, clip]{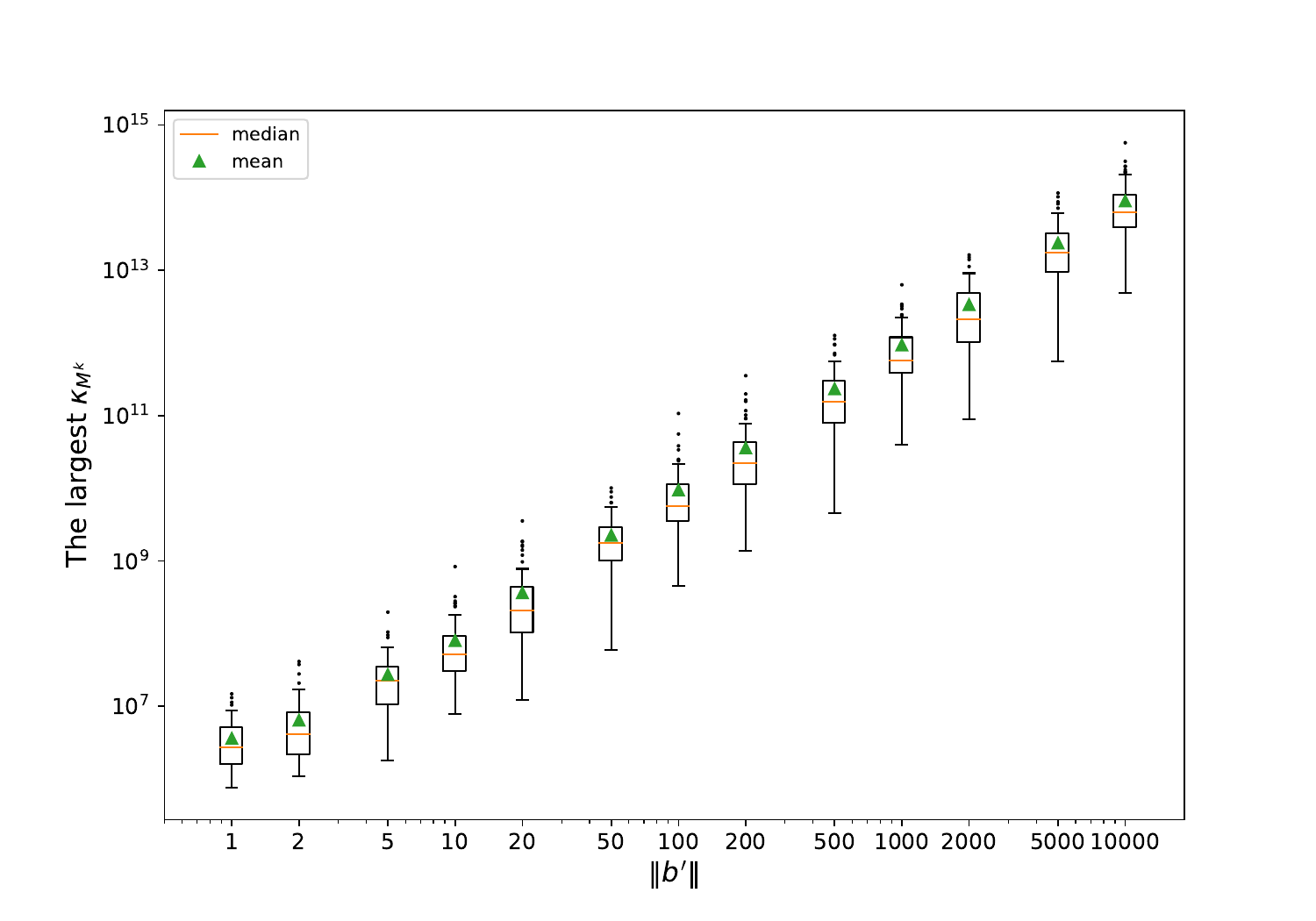}}
\caption{ The effect of different properties on the largest condition number of the OSS in IF-IPM. } \label{fig: largest condition number of the self-dual embedding IF-IPM}
\end{figure}

As we can see in Fig.~\ref{fig: sys} that the condition number of the OSS system is as good as the one for the full Newton system, and better than the Augmented system and Normal Equation system. Also, Figure~\ref{fig: primal_degenerate} verifies that the condition number of the OSS system will go to infinity with rate $\frac{1}{\mu}$ in the worst case, which is better than the Augmented system and Normal Equation System. Figure~\ref{fig: ir_cond_improvement} shows how Iterative Refinement can help to avoid the growing condition number of the Newton system. More precisely, by restarting QIPM in each iteration of IR, the algorithm starts with a system with a relatively low condition number. 
%
\begin{figure}[] \centering
\subfloat[Well-conditoned $A$]{\label{fig: well_cond}\includegraphics[width =0.48 \linewidth, trim = {.1cm .1cm .1cm .1cm}, clip]{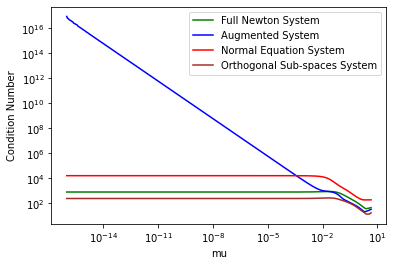}}
\subfloat[Ill-conditoned $A$]{\label{fig: ill_cond}\includegraphics[width =0.48 \linewidth, trim = {.1cm .1cm .1cm .1cm}, clip]{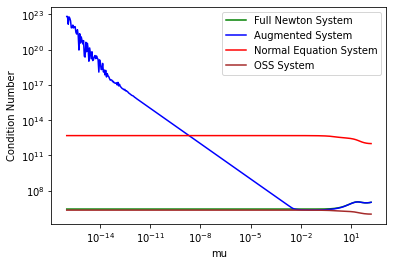}}\\
\subfloat[Primal degenerate problem]{\label{fig: primal_degenerate}\includegraphics[width =0.48 \linewidth, trim = {.1cm .1cm .1cm .1cm}, clip]{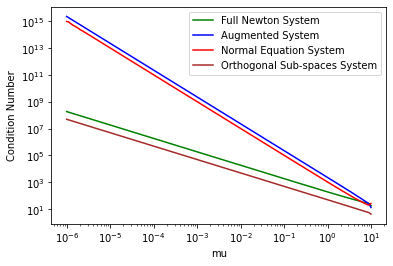}}
\subfloat[Effect of Iterative Refinement]{\label{fig: ir_cond_improvement}\includegraphics[width =0.5 \linewidth, trim = {.1cm .1cm .1cm .1cm}, clip]{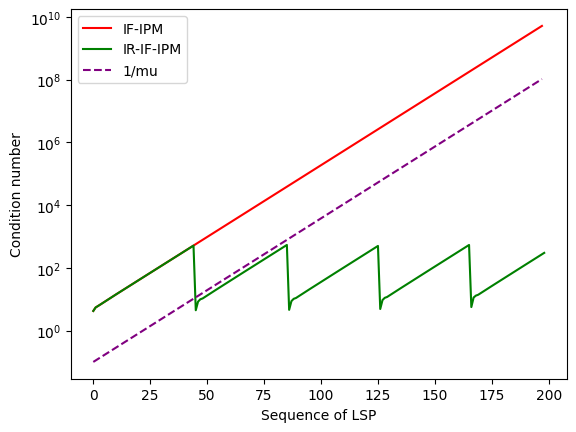}}
\caption{Iterations of IF-IPM with different levels of error.} 
\label{fig: sys}
\end{figure}

To evaluate the performance of the classical IR-IF-IPM, we utilized SciPy’s CG solver and compared its efficiency against IR-IF-IPM with SciPy’s LU decomposition, as well as IR-II-IPM using CG and Cholesky decomposition. Table~\ref{tab:performance} presents the runtime and iteration counts for selected problems from the Netlib library. The results highlight the potential of the proposed IR-IF-IPM in solving large-scale LO instances more efficiently than other inexact IPMs. It is important to note that the quantum version, IR-IF-QIPM, could not solve any Netlib problems on the Qiskit simulator due to the inherent size limitations of the simulator.
\begin{table}[]
    \centering
    \scriptsize
    \renewcommand{\arraystretch}{1.2}
    \resizebox{\textwidth}{!}{\begin{tabular}{lrr|rr|rr|rr|rr}
        \toprule
        \multirow{2}{*}{Name} & \multirow{2}{*}{$m$} & \multirow{2}{*}{$n$} &  \multicolumn{2}{c|}{IR-II-IPM(CGM)} & \multicolumn{2}{c|}{IR-IF-IPM(CGM)} & \multicolumn{2}{c|}{IR-II-IPM(Chol)} & \multicolumn{2}{c}{IR-II-IPM(LU)} \\
        & & &  Iter. & Sec. & Iter. & Sec. & Iter. & Sec. & Iter. & Sec. \\
        \midrule
        25fv47 & 1856 & 3712 &  105 & 1512.84 & 81 & 1026.14 & 103 & 1478.27 & 79 & 1192.39 \\
        adlittle & 19 & 278 &  23 & 0.38 & 22 & 0.33 & 22 & 0.35 & 22 & 0.36 \\
        afiro & 53 & 106 &  16 & 0.06 & 16 & 0.04 & 16 & 0.04 & 16 & 0.04 \\
        agg & 591 & 1182 &  48 & 29.35 & 31 & 15.63 & 44 & 26.38 & 32 & 21.45 \\
        agg2 & 78 & 15 &  44 & 51.71 & 33 & 34.35 & 41 & 47.78 & 32 & 37.92 \\
        agg3 & 758 & 15 &  60 & 70.67 & 38 & 37.25 & 56 & 65.33 & 36 & 47.96 \\
        bandm & 418 & 836 &  53 & 12.66 & 38 & 7.69 & 50 & 12.12 & 37 & 9.39 \\
        beaconfd & 222 & 444 &  25 & 1.20 & 16 & 1.09 & 22 & 1.04 & 20 & 0.82 \\
        blend & 116 & 232 &  20 & 0.28 & 19 & 0.18 & 20 & 0.24 & 19 & 0.20 \\
        bn11 & 1550 & 3100 &  136 & 1164.02 & 87 & 654.98 & 126 & 1081.27 & 83 & 941.94 \\
        \midrule
        Avg. & & &  46.88 & 225.33 & 34.62 & 153.85 & 44.23 & 214.55 & 34.60 & 201.86 \\
        \bottomrule
    \end{tabular}}
    \caption{Performance comparison of IPMs for selected LO problems of Netlib. (Precision $=10^{-6}$)}
    \label{tab:performance}
\end{table}

In conclusion, we examined the potential of the proposed IF-IPM and the challenges of implementing its quantum version on current quantum devices and simulators. Notably, the proposed IR-IF-QIPM is well-suited for fault-tolerant quantum computers. While current NISQ devices are not yet viable for our approach, recent advances have brought fault-tolerant quantum computing closer to realization \cite{rodriguez2024experimental, neven2024meet}. Moreover, GPU-based quantum simulators, such as those on the CUDA-Quantum platform, are emerging as efficient tools for quantum algorithm testing \cite{kim2023cuda}. Future research should explore the practical advantages of IR-IF-QIPMs using these emerging and fast-evolving quantum technologies.

\section{Conclusion} \label{sec: conclusion}
Motivated by the efficient use of QLSA in IPMs, an Inexact Feasible IPM (IF-IPM) is developed with $\Ocal(\sqrt{n}L)$ iteration complexity, analogous to the best exact feasible IPM. 
In terms of classical computing, as well as quantum computing, it is a novel algorithm. 
The improvement in total complexity comes from taking feasible steps using fast but inexact quantum or classical linear solvers. 
We proposed a new linear system, called Orthogonal Spaces System, to generate inexact but feasible Newton steps. 
In consequence, an Inexact Feasible Quantum IPM is developed to solve LO problems nearly better than previous classical and quantum IPMs. 
We analyzed the proposed IF-IPM theoretically and empirically. 
It is necessary to use an iterative refinement scheme to avoid exponential complexity for finding an exact optimal solution using IF-QIPM coupled with QLSAs or IF-IPM with CGMs. 
%


\begin{table}[]
    \centering
\resizebox{\textwidth}{!}{    \begin{tabular}{ |c|c|c| } 
 \hline
 Algorithm & 

 Time Complexity & Comment\\ 
 \hline
Best classical bound & $\Ocal(n^3L)$&\\
Classical II-IPM+CGM \cite{Monteiro2003_Convergence}& $\tilde{\Ocal}_{\phi}\big(n^5L\bar{\chi}^2\big)$&$\bar{\chi}=\max_{B}\{\|A_B^{-1}A\|_F\}$\\
QIPM of~\cite{kerenidisParkas2020_quantum} & 
$\tilde{\Ocal}_{n}\Big(n^{2}L\bar{\kappa}^3 2^{4L}\Big)^*$&\\
QIPM of~\cite{Casares2020_quantum}& 
$\tilde{\Ocal}_{n}\big(n^2L \bar{\kappa}^2 2^{2L}\big)^*$&\\
IR-II-QIPM of~\cite{mohammadisiahroudi2022efficient}& $\tilde{\Ocal}_{n,,\phi,\kappa_{\Ahat}}\big(n^4L\phi\kappa_{\Ahat}^4\big)$&$\phi=\omega^{19}(\|\Ahat\|+\|\bhat\|)$ \\
IR-II-IPM + CGM of~\cite{mohammadisiahroudi2022efficient}& $\tilde{\Ocal}_{\omega,\|\Ahat\|,\|b\|}\big(n^4L\kappa_{\Ahat}\omega^4\big)$& \\
\hline
IR-IF-QIPM &
$\tilde{\Ocal}_{n,\kappa_Q,\omega,\mu^0,L}\big(n^{2.5}L\kappa_Q^2 \|Q\|\omega^5\big)$&\\
IR-IF-IPM + CGM &
$\tilde{\Ocal}_{\mu^0,L}\big(n^{2.5}L\kappa_{Q}\omega^2\big)$&\\
\hline
\end{tabular}}
    \caption{Time complexity of finding the exact solution using different QIPMs ($*$ indicates the time complexity is not attainable)}
    \label{tab:QIPMs complexity}
\end{table}

Table~\ref{tab:QIPMs complexity} indicates that with respect to dimension the best theoretical bound for solving LO problems has been improved for the first time, but this complexity still depends on constants, such as $\kappa_Q$ and $\omega$. 
The proposed IR-IF-QIPM has much better time complexity than IR-II-QIPM for solving LO problems. 
The QIPMs proposed in~\cite{kerenidisParkas2020_quantum, Casares2020_quantum} seem to have a better dependence on $n$. 
However, their time complexities can not be attained since they are based on the premise that QLSAs can provide an exact solution.
This fundamental assumption invalidates the whole convergence of a QIPM algorithm since quantum algorithms are inherently noisy.
All in all, the proposed IR-IF-QIPM has the best complexity among convergent QIPMs and w.r.t the dimension, better complexity than classical IPMs.

As evidenced by the superior complexity of the proposed IR-IF-QIPM, the proposed framework has high adaptability to quantum computers because the noise and error coming from quantum computation are handled efficiently in a way that at each iteration of IPM feasibility of the Newton step is guaranteed, and the noise is added to the duality gap. Consequently, as Figure~\ref{fig: eta_iter} illustrates, the algorithm is highly robust against the noise since at each iteration of the IF-IPM, the Newton direction decreases the optimality gap and addresses the error of the previous direction. On the other hand, the novel OSS system has a better condition number bound, which makes it more adaptable for QLSAs. By augmenting iterative refinement, our approach leads to the best complexity among other QIPMs.

Although this paper studied an application of QLSAs, the proposed method takes advantage of the best iteration complexity of feasible IPMs and the low cost of iterative methods, such as CGM.
The condition number of OSS is increasing more slowly than that of other Newton systems. 
We also employed an iterative refinement scheme using the proposed IF-QIPM with low precision to address both errors of QLSAs and the growing condition number of Newton systems. 
Another promising line of research is to study preconditioning and regularization techniques for the OSS to mitigate the impact of the growing condition number in IF-(Q)IPMs through iterative methods. 
This paper is the first comprehensive approach to develop an inexact but feasible IPMs to solve LO problems. 
This direction can be pursued by modifying the NES to guarantee the feasibility of the inexact solution but taking advantage of small positive definite coefficient matrices. 
The proposed IF-IPM can also be extended to other optimization problems, such as conic and nonlinear optimization problems.

{\hyphenpenalty=100000
\bibliographystyle{dependencies/siamplain}

\bibliography{dependencies/references.bib}}
\end{document}